\documentclass{article}
\usepackage[utf8]{inputenc}
\usepackage[backend=biber, style=numeric]{biblatex}
\usepackage{amsfonts}
\usepackage{amsthm}
\usepackage{amsmath}
\usepackage{biblatex}
\usepackage{setspace}
\usepackage{verbatim}
\usepackage{tikz-cd}
\usepackage{enumitem}
\usepackage{tensor}
\usepackage{pifont}
\usepackage{enumitem}
\usepackage{txfonts}
\usepackage{mathtools}
\usepackage{stmaryrd}
\usepackage[T1]{fontenc}
\usepackage[a4paper, total={6in, 8in}]{geometry}

\title{A Graphical Calculus for Classical and Quantum Microformal Morphisms}
\author{Andreas Swerdlow \\ \small The University of Manchester, UK \\ \small\texttt{andreas.swerdlow@manchester.ac.uk}}
\date{}

\newtheorem{theorem}{Theorem}[section]

\newtheorem{proposition}[theorem]{Proposition}
\newtheorem{lemma}[theorem]{Lemma}
\newtheorem{definition}[theorem]{Definition}
\newtheorem{example}[theorem]{Example}
\newtheorem{remark}[theorem]{Remark}

\newtheorem*{pullback_law}{The Pullback by a Thick Morphism}
\newtheorem*{composition_law}{The Composition Law for Thick Morphisms}
\newtheorem*{transformation_law}{The Transformation Law for Thick Morphisms}
\newtheorem*{general_eq}{The General Equation}

\newtheorem*{q_pullback_law}{The Pullback by a Quantum Thick Morphism}
\newtheorem*{q_composition_law}{The Composition Law for Quantum Thick Morphisms}
\newtheorem*{q_transformation_law}{The Transformation Law for Quantum Thick Morphisms}
\newtheorem*{q_general_eq}{The Quantum General Equation}

\newtheorem*{Feynman_rules}{The Feynman Rules}

\newtheorem*{outline}{Outline of the Paper}
\newtheorem*{acknowledgement}{Acknowledgement}

\newcommand{\dely}[1]{\frac{\partial}{\partial y^{#1}}}

\newcommand{\contraction}[3]{\langle #1 \mathopen{\Yleft} #3 \mathclose{\Yright} #2 \rangle}
\newcommand{\expih}[1]{e^{\frac{i}{\hbar} #1}}

\tikzstyle{vertex}=[circle, draw, inner sep=0pt, minimum size=6pt]
\newcommand{\vertex}{\node[vertex]}

\begin{document}

\maketitle

\begin{abstract}
    We develop a graphical calculus for the microformal or thick morphisms introduced by Th. Voronov. This allows us to write the infinite series arising from pullbacks, compositions, and coordinate transformations of thick morphisms as sums over bipartite trees. The methods are inspired by those employed by Cattaneo-Dherin-Felder in their work on formal symplectic groupoids. We also extend this calculus to quantum thick morphisms, which are special types of Fourier integral operators quantizing classical thick morphisms. The relationship between the calculi for classical and quantum thick morphisms resembles the relationship between the semi-classical and full perturbative expansions over Feynman diagrams in quantum field theory.
\end{abstract}

\tableofcontents

\addcontentsline{toc}{section}{Introduction}
\section*{Introduction}\label{sec_intro}

In \cite{Voronov_pullbacks}, Th. Voronov introduced the notion of microformal or thick morphisms, which generalises the notion of smooth maps between manifolds, and used them to construct $L_\infty$-morphisms for functions on homotopy Poisson and homotopy Schouten manifolds. A thick morphism is a special type of formal canonical relation between the cotangent bundles of two manifolds, which is thought of as a map between the underlying manifolds themselves. The key feature of thick morphisms is that they give rise to a generalisation of the pullback on functions, which is now nonlinear and formal. 

Pullbacks by thick morphisms are defined using an equation of the fixed point type, which can be solved (in principle) by iterations. However, since this is an unending process, it does not give a closed form for the full power series expansion of the pullback, and can only practically supply approximate solutions to low orders. Compositions and coordinate transformations of thick morphism generating functions are defined by the same equation as pullbacks so are also formal and suffer from the same issues. In this paper, we obtain the full solutions for pullbacks, compositions, and coordinate transformations as expansions over a class of bipartite trees. We call the resulting calculational process the graphical calculus for thick morphisms.  

This graphical calculus is heavily inspired by a graphical calculus of Cattaneo, Dherin and Felder (CDF), first presented in \cite{groupoid} and later in \cite{operad}. In \cite{groupoid}, CDF consider (formal) symplectic groupoids, with multiplicative structures determined by a specific class of generating functions. They develop a graphical calculus to obtain a perturbative condition for the associativity of the resulting groupoid multiplication. In \cite{operad}, CDF describe an operad, with objects given by these generating functions, and a composition map which captures the composition of the resulting groupoid products, and they extend their graphical calculus from \cite{groupoid} to the whole operad. Crucially for this paper, the equation which defines CDFs composition map is essentially the composition equation for generating functions of thick morphisms, the only difference being the forms of the power series for the generating functions. We have made a few changes to adapt CDFs method to our case, some of which can be seen as generalisations and some of which can be seen as modifications. However, many of the arguments and calculations made are analogous to those in \cite{operad}. 

CDF based their graphical calculus on a standard approach in Numerical Analysis to the study of order conditions of Runge-Kutta methods. The concepts they use, such as the use of trees to keep track of terms coming from Taylor expansions, the symmetry factor on rooted trees, and the Butcher product are all established tools in the theory of B-series (see \cite{B-series} and Chapter 3 of \cite{GNI_book}). This theory is used to compare the Taylor series for an exact solution to an ODE with the Taylor series of a numerical solution obtained by a Runge-Kutta method, by writing both series as expansions over trees. In CDFs case they are only interested in the exact solution, and the graphical calculus they obtain can be seen as an application of this standard method, as applied to a partitioned system of ODEs (where the theory is now called the theory of P-series). Indeed, where a general partitioned system of the form 
\begin{equation*}
    \dot{y} = f(y,z), \quad \dot{z} = g(y,z) 
\end{equation*}
gives rise to expansions for the exact solutions $y, z$ in terms of bicoloured trees (trees where each vertex is assigned one of two colours), the partitioned systems considered by CDF and this paper are of the form
\begin{equation*}
    \dot{y} = f(z), \quad \dot{z} = g(y), 
\end{equation*}
which imposes bipartiteness (edges can only connect vertices of different colour) on the trees in the expansions of the exact solutions.

In \cite{Voronov_oscill} and \cite{Voronov_microformal_homotopyalg} Voronov also introduces the notion of quantum thick morphisms, which are exactly the quantum analogue of thick morphisms (which we will call classical thick morphisms when the distinction is required). Indeed the category of classical thick morphisms is the $\hbar \rightarrow 0$ limit of the category of quantum thick morphisms, which is defined as the opposite category of the category with morphisms given by a certain type of Fourier integral operators. Each quantum thick morphisms also has a "quantum" generating function, and as for the classical case, pullbacks, compositions and coordinate transformations are defined in terms of this generating function and are formal. 

We would again like to have closed forms for the resulting formal series, which we achieve using a graphical calculus which is exactly the quantum analogue of the graphical calculus for classical thick morphisms. In particular the the quantum calculus extends the classical calculus from trees to graphs with loops, where each loop adds a factor of $\hbar$. CDF did not consider a quantum case in their papers, so we believe that the work in this section is original. 

For other works involving the study of thick morphisms, see \cite{KHUDAVERDIAN_VORONOV}, \cite{VORONOV2024105105}, \cite{SHEMYAKOVA_VORONOV}, and \cite{Shemyakova_Yilmaz}.  Besides the work of CDF, graphical calculi have been put to use many other places in geometry. Some examples include Maxim Kontsevich's deformation quantization formula \cite{Kontsevich2003}, which involves a sum over so-called Kontsevich graphs, and Ezra Getzler's formula for the generalized Campbell-Hausdorff series for nilpotent $L_\infty$-Algebras, which involves a sum over rooted trees \cite{Getzler}.

\begin{outline}\normalfont
    We start in section \ref{sec_prelim} by recalling the definitions for classical thick morphisms, and the equations defining their pullbacks, compositions and coordinate transformations. Since each equation is of the same form we can treat them together by writing down a general equation. Motivated by the low order "by-hand" calculation in section \ref{sec_low_order}, we next propose an expansion for the solution to this general equation in section \ref{sec_trees_and_Feynman_rules} and prove it in section \ref{sec_gen_eq_solve}. We then use this general graphical calculus to obtain specialised graphical calculi for pullbacks, compositions, and coordinate transformations in sections \ref{sec_pullback_graph}, \ref{sec_comp_graph}, and \ref{sec_transformation_graph}.

    In section \ref{sec_quantum_graphical_calc} we perform the same process again for quantum thick morphisms, by first finding a general graphical calculus in section \ref{sec_qgen_eq_solve}, and then specialising to the cases of pullbacks, compositions, and coordinate transformations in sections \ref{sec_qpullback_graph}, \ref{sec_qcomp_graph}, and \ref{sec_qtransform_graph}. 

    Finally, in section \ref{sec_discussion}, we briefly discuss how the classical calculus can be obtained as the classical limit of the quantum calculus, and how to extend both calculi to the case of supermanifolds. We finish by demonstrating the graphical calculus with a low order calculation of a quantum pullback. 
\end{outline}

\begin{acknowledgement}
    I would like to thank my supervisor Ted Voronov for suggesting the line of research that became this paper, and for his guidance. Ted informed me that the possibility of a graphical calculus for thick morphisms based on bipartite trees was suggested to him and Hovhannes Khudaverdian by Sergei Shadrin, in an email discussion following a talk by Khudaverdian in June 2020.
\end{acknowledgement}

\section{The Graphical Calculus for Classical Thick Morphisms}\label{sec_graphical_calc}

\subsection{Preliminaries on Classical Thick Morphisms}\label{sec_prelim}

We start by recalling the definition of a thick morphism as given in \cite{Voronov_microformal_homotopyalg}. 
\begin{definition}
    Let $M_1$ and $M_2$ be supermanifolds. A thick morphism $\Phi : M_1 \text{\ding{225}} M_2$ is a formal canonical relation $\Phi \subset T^*M_2 \times \overline{T^*M_1}$, with a collection of generating functions for $\Phi$ in each local coordinate system, that depend only on the position variables of $M_1$ and the momentum variables of $M_2$. That is, given any local coordinates $x^a$ and $y^i$ on $M_1$ and $M_2$, with corresponding conjugate momenta $p_a$ and $q_i$ (cotangent fibre coordinates), there is a function $S = S(x,q)$ defined in these coordinates, such that $\Phi$ is given locally as a subset by
    \begin{equation*}
        \Phi = \left\{  (y^i, q_i ~ ; ~ x^a, p_a) \in T^*M_2 \times \overline{T^*M_1} ~~ \middle| ~~ y^i = (-1)^{\tilde{i}} \frac{\partial S}{\partial q_i}(x,q), ~~ p_a = \frac{\partial S}{\partial x^a}(x,q)  \right\} .
    \end{equation*}
    By formal canonical relation, we mean a Lagrangian submanifold of the formal neighbourhood of $M_2 \times T^*M_1$ in $T^*M_2 \times \overline{T^*M_1}$. Therefore, the generating functions $S(x,q)$ are taken to be formal power series 
    \begin{equation*}
        S(x,q) = S^0(x) + S^i(x)q_i + S^{ij}(x)q_jq_i + S^{ijk}(x)q_kq_jq_i + \dots
    \end{equation*}
    in the momentum variables. We will always assume without loss of generality that the coefficients $S^{a_1 \dots a_m}(x)$ are symmetric in their indices.
\end{definition}

From here on out we will only consider the setting of regular (non-super) manifolds. We do this for simplicity of exposition, and leave the extension to the super case until the end of the paper.
 
We now recall the equations governing the three processes we are interested in: pullbacks, compositions, and changes of coordinates. In each case let $M_1, M_2, M_3$ be manifolds, and let $\Phi : M_1 \text{\ding{225}} M_2$, $\Phi_F: M_1 \text{\ding{225}} M_2$ and $\Phi_G: M_2 \text{\ding{225}} M_3$ be thick morphisms between them. Choosing suitable local coordinates $x, y, z$ on $M_1, M_2, M_3$ respectively with corresponding conjugate momenta $p, q, r$, we have the following generating functions for $\Phi$, $\Phi_F$ and $\Phi_G$ given as formal power series in the momenta.
\begin{align*}
    S(x,q) & = S_0(x) + \varphi^i(x)q_i + S^{ij}(x)q_iq_j + \dots \\
    F(x, q) & = F_0(x) + \varphi^i(x)q_i + F^{ij}(x)q_iq_j + \dots \\
    G(y,r) & = G_0(y) + \gamma^i(y)r_i + G^{ij}(y)r_ir_j + \dots 
\end{align*}
Given these generating functions we then have the following equations.
\begin{pullback_law}
    The pullback by $\Phi$ is a formal mapping $\Phi^*: C^\infty (M_2) \rightarrow C^\infty (M_1)$ defined for any $g \in C^\infty (M_2)$ by 
    \begin{equation}\label{eq_pullback_1}
        \Phi^*[g](x) = g(\overline{y}) + S(x,\overline{q}) - \overline{y}^\alpha \overline{q}_\alpha ,
    \end{equation}
    where $\overline{y}=\overline{y}(x)$ and $\overline{q}=\overline{q}(x)$ are determined by the system 
    \begin{align}\label{eq_pullback_2}
        \overline{q}^\alpha = \frac{\partial}{\partial y^\alpha} g(\overline{y}), ~~~~ \text{and} ~~~~
        \overline{y}^\alpha = \frac{\partial}{\partial q^\alpha} S(x,\overline{q}).
    \end{align}
\end{pullback_law}

\begin{composition_law}
    The composition $\Phi_H = \Phi_G \circ \Phi_F: M_1 \text{\ding{225}} M_3$ is also a thick morphism and has generating function $H(x,r)$ given by
    \begin{equation}\label{eq_thick_comp_1}
        H(x,r) = G(\overline{y},r) + F(x,\overline{q}) - \overline{y}^\alpha\overline{q}_\alpha,
    \end{equation}
    where $\overline{y}=\overline{y}(x,r)$ and $\overline{q}=\overline{q}(x,r)$ are the unique power series in $r$ determined by the system 
    \begin{align}\label{eq_thick_comp_2}
        \overline{q}^\alpha = \frac{\partial}{\partial y^\alpha} G(\overline{y},r), ~~~~ \text{and} ~~~~
        \overline{y}^\alpha = \frac{\partial}{\partial q^\alpha} F(x,\overline{q}).
    \end{align}
\end{composition_law}

\begin{transformation_law}
    Let $x = x(x')$, $y = y(y')$ be an invertible change of local coordinates, with $p', q'$ the corresponding conjugate momenta. If $S(x,q)$ is the generating function for $\Phi$ in the "old" coordinate system $(x,p,y,q)$, then the generating function $S'(x',q')$ in the "new" coordinate system $(x',p',y',q')$ is given by 
    \begin{equation}\label{eq_thick_transformation_1}
        S'(x',q') = S(x(x'),\overline{q}) + y^{\alpha'}(\overline{y})q_{\alpha'} - \overline{y}^\alpha \overline{q}_\alpha ,
    \end{equation}
    where $y^{\alpha'} = y^{\alpha'}(y)$ is the inverse change of coordinates, and $\overline{y}=\overline{y}(x',q')$ and $\overline{q}=\overline{q}(x',q')$ are the unique power series in $q'$ determined by the system 
    \begin{align}\label{eq_thick_transformation_2}
        \overline{q}^\alpha = \frac{\partial}{\partial y^\alpha} y^{\alpha'}(\overline{y}), ~~~~ \text{and} ~~~~
        \overline{y}^\alpha = \frac{\partial}{\partial q^\alpha} S(x(x'),\overline{q}).
    \end{align}
\end{transformation_law}

All three of these equations take essentially the same form, so we now write down an abstract general equation capturing this form, which we will solve using a graphical calculus described in section \ref{sec_trees_and_Feynman_rules}. 
\begin{general_eq}
    Consider a formal power series $S(q) = S_0 + \varphi^iq_i + S^{ij}q_iq_j+ \dots$, and a smooth function $G(y) : \mathbb{R}^{d} \rightarrow \mathbb{R}$. The coefficients of $S$ and the function $G$ may depend on other parameters which we ignore to simplify notation (in our cases the coefficients of $S$ depend on $x$ and $G$ is a formal power series in either $r$ or $q_{\alpha'}$). We are interested in the result $R(S,G)$ of the equation 
        \begin{equation}\label{eq_general_eq_1}
            R(S,G) = S(\overline{q}) + G(\overline{y}) - \overline{y}^\alpha \overline{q}_\alpha , 
        \end{equation}
        where 
        $\overline{y}$ and $\overline{q}$ are the unique functional power series in $G$ determined by the system
        \begin{align}\label{eq_general_eq_2}
        \overline{q}^\alpha = \frac{\partial}{\partial y^\alpha}  G(\overline{y}), ~~~~ \text{and} ~~~~
        \overline{y}^\alpha = \frac{\partial}{\partial q^\alpha} S(\overline{q}).
    \end{align}
\end{general_eq}

\subsection{Low Order Calculation}\label{sec_low_order}

To see what the graphical calculus should look like, let's try to solve equations \eqref{eq_general_eq_1} and \eqref{eq_general_eq_2} by iterations using Taylor's theorem. We can only do this to low orders since the terms quickly become unmanageable. However, the structure of these low order terms will be instructive in developing a procedure for calculating terms of all orders. 

We calculate up to second order in $G$, i.e. we ignore any terms containing more than one factor coming from $G$ or its derivatives. It turns out that we will only need the terms in $S$ up to second order in $q$, so we start with 
\begin{align*}
    S(q) = S_0 + \varphi^iq_i + S^{ij}q_iq_j .
\end{align*}
We write $\varphi$ for the vector with coefficients $\varphi^i$. Then by taking derivatives we get 
\begin{equation*}
    \overline{q}^\alpha = \frac{\partial}{\partial y^\alpha}  G(\overline{y})
\end{equation*}
and
\begin{equation*}
    \overline{y}^\alpha = \frac{\partial}{\partial q^\alpha} S(\overline{q}) = \varphi^\alpha + 2  S^{\alpha b} \overline{q}_b.
\end{equation*}
Next we plug our expression for $\overline{q}_\alpha$ into our expression for $\overline{y}^\alpha$, which gives
\begin{align*}
    \overline{y}^\alpha & = \varphi^\alpha + 2 S^{\alpha b} \dely{b} G(\overline{y})   
\end{align*}
Then we plug this expression into itself and expand about $\varphi$ using Taylor's formula (here we will only need the formula up to linear terms), giving 
\begin{gather*}
    \overline{y}^\alpha = \varphi^\alpha + 2  S^{\alpha b} \dely{b} G(\varphi + \dots)  \\
     = \varphi^\alpha + 2  S^{\alpha b} \dely{b} G(\varphi)  + 4 S^{\alpha b} \dely{a}\dely{b} G(\varphi) S^{a c} \dely{c} G(\varphi) .
\end{gather*}
Note that we have performed a second Taylor expansion about $\varphi$ for the last term, only needing the zeroth order in these cases. We can now plug this expression into our expression for $\overline{q}_\alpha$ and again expand about $\varphi$ to get
\begin{gather*}
    \overline{q}^\alpha =  \dely{\alpha} G(\varphi + \dots)  \\
    =  \dely{\alpha} G(\varphi) + 2 \dely{b} G(\varphi) S^{a b} \dely{a}\dely{\alpha} G(\varphi) .
\end{gather*}
We now find an expansion for the terms in equation \eqref{eq_general_eq_1}. To calculate $ G(\overline{y})$ we simply perform the same process as for $\overline{q}_\alpha$, which results in a very similar expansion but with any $\dely{\alpha}$ derivatives removed. So we have
\begin{gather*}
     G(\overline{y}) =  G(\varphi) + 2  \dely{b} G(\varphi) S^{a b} \dely{a} G(\varphi) .
\end{gather*}
To calculate $S(\overline{q})$ we just plug in our expansion for $\overline{q}$ and expand the products, giving
\begin{gather*}
    S(\overline{q}) = S_0 +  \dely{b} G(\varphi) \varphi^{b} + 2 \dely{c} G(\varphi) S^{a c} \dely{a}\dely{b} G(\varphi) \varphi^{b} \\ 
        + \dely{a}G(\varphi) \dely{b} G(\varphi) S^{a b}     .
\end{gather*}
Finally we expand the product $\overline{y}^\alpha\overline{q}_\alpha$ ignoring the higher order terms and get 
\begin{gather*}
    \overline{y}^\alpha\overline{q}_\alpha =  \varphi^\alpha \dely{\alpha} G(\varphi) + 2  \varphi^\alpha \dely{b} G(\varphi) S^{a b} \dely{a}\dely{\alpha} G(\varphi) \\   + 2 \dely{b} G(\varphi) S^{\alpha b} \dely{\alpha} G(\varphi) . 
\end{gather*} 
All of these term cancel with terms in $G(\overline{y})$ and $S(\overline{q})$ and we end up with the much simpler expansion for $R(S,G)$ given by
\begin{gather*}
    R(S,G) = S_0 + G(\varphi) 
    + S^{a b} \dely{a}G(\varphi) \dely{b} G(\varphi)  .
\end{gather*}

First notice that all terms in the above expansion and in the expansions for $G(\overline{y})$, $S(\overline{q})$, and $\overline{y}^\alpha\overline{q}_\alpha$ are contractions of upper indices coming from coefficients in $S$ with lower indices coming from derivatives of $G$. A standard approach to representing contractions of indices is via graphs, where an edge between two vertices represents a contraction between the objects represented by the vertices. We would like to be able to write the expansion of $R(S,G)$ up to any order as a sum over terms coming from a certain class of graphs, thereby obtaining a combinatorial approach to calculating pullbacks, compositions, and coordinate transformations of thick morphisms. 

To see what this class of graphs should look like, notice that we need two types of vertices, one to represent coefficient functions in $S$ and another to represent derivatives of $G$. Therefore, we will use bipartite graphs, where each vertex is coloured black or white. Also notice that the expansions for $\overline{q}_\alpha$ and $\overline{y}^\alpha$ contain terms of the same form except that one index is left uncontracted. For $\overline{q}_\alpha$ the uncontracted lower index comes from a derivative of $G$ and for $\overline{y}^\alpha$ the uncontracted upper index comes from a coefficient in $S$. We can represent an uncontracted index by an unconnected edge coming from one vertex. This is the same as choosing a what is called a root for each tree. In the next section we formally define our class of graphs and the terms associated to them. 

\subsection{Trees and Feynman Rules}\label{sec_trees_and_Feynman_rules}

Most of the conventions and notations here are taken from \cite{operad} with some new additions.
\begin{definition}
    \begin{itemize}
        \item A bipartite graph $t$ is a triple $t = (B_t, W_t, E_t)$, where $B_t = \{1, \dots n\}$ and $W_t = \{-1,\dots,-m\}$ are sets of disjoint vertices and $E_t \subset B_t \times W_t$ is the submultiset of edges (so we allow more than one edge to connect two vertices, but do not distinguish such edges). We will draw elements of $B_t$ as $\bullet$ and elements of $W_t$ as $\circ$. In a bipartite graph, vertices can only connect to vertices of the opposite colour. We denote by $V_t = B_t \cup W_t$ the set of all vertices. For any edge $e \in E_t$, denote by $e_\bullet$ its component in $B_t$ and by $e_\circ$ its component in $W_t$ (the black white vertices to which $e$ is attached).
        \item An isomorphism of bipartite graphs $t, t'$ with the same number of black and white vertices (denoted by $|t|_\bullet = |t'|_\bullet = n$ and $|t|_\circ = |t'|_\circ = m$ respectively) consists of a pair of permutations $(\sigma_\bullet,\sigma_\circ) \in S_n \times S_m$ and a bijection $\rho: E_t \rightarrow E_{t'}$ such that for all $e \in E_t$ we have $\rho(e) = ((\sigma_\bullet(e_\bullet),\sigma_\circ(e_\circ)))$.
        \item A bipartite tree is a bipartite graph $t$ which has no cycles. An isomorphism of bipartite trees is just an isomorphism of bipartite graphs.
        \item A rooted bipartite tree is a bipartite tree with a distinguished vertex called the root. An isomorphism of rooted bipartite trees is an isomorphism of bipartite trees $t \xrightarrow{\cong} t'$ that sends the root of $t$ to the root of $t'$. 
        \item A hooked bipartite tree is just a rooted bipartite tree where we think of the root having one "half-connected" edge which comes from it that is not connected to any other vertex.
    \end{itemize}
\end{definition}

We denote by $sym(t)$ the group of automorphisms of $t$, where the notion of isomorphism used will be clear from context. We denote the set of bipartite trees by $T$, the set of rooted bipartite trees by $RT$, and the set of hooked bipartite trees by $HT$. For any set of trees $A$, we use brackets $[A]$ to denote the trees up to isomorphism. We also use subscripts $RT_\bullet$ and $RT_\circ$ to denote the subsets of rooted trees with black and white roots respectively and similarly with hooked trees. We use the symbols $\bullet$, $\circ$ to denote the single vertex trees, and $\multimapdotinv$, $\multimapinv$ for the single vertex hooked trees.

The leaves of a tree are the vertices with degree one (connected to exactly one other vertex), discounting the possible root. We use the superscript $A^\circ$ to denote the trees in $A$ with white leaves only. Note that the tree with single black vertex $\bullet$ is contained in $RT^\circ$, $T^\circ$ and $\multimapdotinv$ is contained in $HT^\circ$ since in each case the black vertex is not a leaf. Finally, we define $A^+ := A - \{\bullet\}$ or $A^+ := A - \{\multimapdotinv\}$ as appropriate.

We now introduce a formula associating an appropriate term of the expansions to each such tree. Note that we will use both notations $S^\alpha$ and $\varphi^\alpha$ for the coefficients of the linear term in $q$ of $S$.  

\begin{definition}
    To each bipartite tree $t = (B_t, W_t, E_t)$ we associate a term $\contraction{S}{G}{t}$ by the following. For each vertex of any colour $v \in V_t$, let $E_v = \{ e \in E_t ~ | ~ \text{$e$ is connected to $v$}\}$ and denote the elements of $E_v$ by $e^v_1, \dots, e^v_{m_v}$, so $m_v$ is the number of edges connected to $v$. We then use the labels $e^v_k$ as dummy indices to be summed over by the Einstein summation convention and define
    \begin{equation}\label{eq_contraction_def}
        \contraction{S}{G}{t} := \left( \prod_{v \in B_t} m_v! S^{e^v_1, \dots, e^v_{m_v}} \right) \left( \prod_{w \in W_t} \dely{e^w_1} \dots \dely{e^w_{m_w}} G(\varphi) \right).
    \end{equation}
    For a rooted tree $t$ we define $\contraction{S}{G}{t}$ by the same formula, simply forgetting the choice of root. However, for a hooked tree $t$ the half-connected edge leads to an uncontracted index in the factor associated to the single vertex connected to this edge. So the resulting term will be a tensor with one extra upper or lower index.
\end{definition}

This formula can be encoded in a set of Feynman-esque rules 

\begin{Feynman_rules}
    To each bipartite tree $t$ we associate a term $\contraction{S}{G}{t}$ by the following process. 
    \begin{enumerate}
        \item For every white vertex $u \in W_t$, write a factor $\dely{*} ... \dely{*} G(\varphi)$, where the number of partial derivatives is the number of edges connected to the vertex. Leave these indices blank for now.
        \item For every black vertex $v \in B_t$ write a factor $m! S^{* \dots *}$, where $m$ is the number of edges connected to the vertex and the number of (for now blank) upper indices of $S$ is equal to $m$. Naturally if $v$ has no edges connected to it write a factor $S_0$.
        \item For every edge between a black and a white vertex, contract a blank index from each of the factors corresponding to the vertices.
    \end{enumerate} 
\end{Feynman_rules}

The main difference with our method and the one presented in \cite{operad} is a change in the factors associated to black vertices. In \cite{operad}, the form for black vertices mirrors that for white vertices but with derivatives of coefficient functions from $S$ instead of $G$. Our method takes advantage of the special form of $S(q)$ as a formal power series in $q$ to expand the black vertex terms. This is why CDF need weightings on black vertices and we do not.

We are now ready to present the expansion for $R(S,G)$ in terms of white-leaved bipartite trees.
\begin{proposition}\label{prop_general_eq_graph}
    Given a formal power series $S(q) = S_0 + \varphi^iq_i + S^{ij}q_iq_j+ \dots$, and a smooth function $G(y) : \mathbb{R}^{d} \rightarrow \mathbb{R}$, then $R(S,G)$, which is defined by equations \eqref{eq_general_eq_1} and \eqref{eq_general_eq_2}, has the following expansion as a functional power series in $G$. 
    \begin{equation}\label{eq_gen_eq_expansion}
        R(S,G) = \sum_{t \in T^\circ} \frac{1}{|t|_\circ! |t|_\bullet!} \contraction{S}{G}{t} = \sum_{t \in [T^\circ]} \frac{1}{|sym(t)|} \contraction{S}{G}{t}.
    \end{equation}
    The second equality holds since isomorphic trees always produce the same term from the Feynman rules, and the number of trees isomorphic to $t$ is equal to $\frac{|t|_\circ ! |t|_\bullet !}{|sym(t)|}$.
\end{proposition}

Before proving the Proposition, we check that it reproduces our low order calculation from the previous section with significantly less effort. 
\begin{example}
    Let's calculate $R(S,G)$ to second order in $G$ using the graphical calculus. There are only three isomorphism classes of white-leaved bipartite trees with up to two white vertices. They are
    
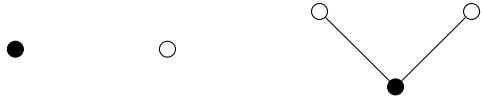
\begin{center}\vspace{-\baselineskip}
    \[\begin{tikzpicture}[x=2cm, y=1cm]
	\vertex[fill] (a) at (0,0) {};
	\vertex (b) at (1,0) {};
	\vertex (c) at (2,0.5) {};
	\vertex[fill] (d) at (2.5,-0.5) {};
        \vertex (e) at (3,0.5) {};
	\path
		(c) edge (d) 
            (d) edge (e)
	;
    \end{tikzpicture}\] 
    \end{center}
   The third tree has one non-trivial symmetry so it's symmetry group has order $2$. So to second order in $G$ we get  
   \begin{equation*}
       R(S,G) \approx S_0 + G(\varphi) + S^{ab} \dely{a} G(\varphi) \dely{b} G(\varphi) .
   \end{equation*}
   This agrees with the calculation in section \ref{sec_low_order}.
\end{example} 

We split the proof of Proposition \ref{prop_general_eq_graph} into the following lemmas. Even with the difference in terms mentioned above, many of the calculations are very similar to those in \cite{operad}. We omit any proofs of lemmas that are completely unchanged. 

We first introduce a way of constructing a new tree from a list of hooked trees which will be crucial to our calculations.
\begin{definition}\label{def_tree_bracket}
    Given a list of black-hooked trees $t_1, \dots t_m \in [HT_\bullet]$, we can construct a new rooted tree, denoted $[t_1,\dots,t_m]_{\circ}$, by connecting all of the half-connected edges of $t_1, \dots , t_m$ to a white vertex and declaring this to be the root. If we instead attach a new half-connected edge to this vertex, we can construct a hooked tree which we denote $[t_1,\dots,t_m]_{\multimapinv}$. The same operations work with black and white vertices swapped. 
\end{definition}
This allows us to recursively construct the set of isomorphism classes of rooted or hooked bipartite trees. Indeed, for any $t \in [RT]$, we have one of three cases.
\begin{enumerate}
    \item $t$ has a single vertex so $t=\bullet$ or $t= \circ$
    \item $t = [t_1, \dots , t_m]_{\circ}$ for some $t_1, \dots, t_m \in [HT_\bullet]$
    \item $t = [t_1,\dots,t_m]_\bullet$ for some $t_1, \dots, t_m \in [HT_\circ]$
\end{enumerate}
Similarly for hooked trees.

We can use this to introduce the symmetry factor $\sigma(t)$ of a rooted or hooked bipartite tree $t$ by the following recursive definition.
\begin{definition}[Symmetry factor]
    \begin{enumerate}
         \item []
         \item $\sigma(\circ) = \sigma(\bullet) = 1$
         \item $\sigma(t) = \mu_1! \mu_2! \dots \sigma(t_1)\dots\sigma(t_m)$ if $t = [t_1, \dots, t_m]_{\circ}$ or $t = [t_1, \dots, t_m]_{\bullet}$,
     \end{enumerate}
     where the $\mu_i$ are the sizes of the partions created by separating the set $\{t_1,\dots,t_m\}$ into isomorphism classes. 

     We call this the symmetry factor, since if we take any representative $t' \in RT$ of the equivalence class $t \in [RT]$, then $\sigma(t)$ recursively calculates the number of symmetries of $t'$, i.e. $\sigma(t) = |sym(t')|$, and similarly for hooked trees. Note that $\sigma(t)$ depends upon the choice of root.
 \end{definition}  

The next lemma calculates the terms associated to trees constructed as in definition \ref{def_tree_bracket}. It is crucial to most of the calculations that follow. The proof is by inspection. 
\begin{lemma}\label{lemma_tree_bracket_calc}
\begin{enumerate}[label=\alph*)]
    \item[]
    \item $\contraction{S}{G}{t} = \contraction{S}{G}{t_1}^{b_1} \dots \contraction{S}{G}{t_m}^{b_m} \frac{\partial}{\partial y^{b_1}} \dots \frac{\partial}{\partial y^{b_m}} G(\varphi)$ \\ if $t = [t_1, \dots, t_m]_{\circ}$
    \item $\contraction{S}{G}{t} = m! S^{b_1 \dots b_m} \contraction{S}{G}{t_1}_{b_1} \dots \contraction{S}{G}{t_m}_{b_m} $ \\ if $t = [t_1, \dots, t_m]_{\bullet}$
    \item $\contraction{S}{G}{t}_\alpha = \contraction{S}{G}{t_1}^{b_1} \dots \contraction{S}{G}{t_m}^{b_m} \frac{\partial}{\partial y^\alpha} \frac{\partial}{\partial y^{b_1}} \dots \frac{\partial}{\partial y^{b_m}} G(\varphi)$ \\ if $t = [t_1, \dots, t_m]_{\multimapinv}$
    \item $\contraction{S}{G}{t}^{\alpha} = (m+1)! S^{\alpha b_1 \dots b_m} \contraction{S}{G}{t_1}_{b_1} \dots \contraction{S}{G}{t_m}_{b_m}$ \\ if $t = [t_1, \dots, t_m]_{\multimapdotinv}$
\end{enumerate}
\end{lemma}

\subsection{Solving the General Equation}\label{sec_gen_eq_solve}

We are now ready to start our calculations by first finding expansions over hooked trees for $\overline{q}$ and $\overline{y}$.
\begin{lemma}
    We have unique solutions $\overline{q}$ and $\overline{y}$ to equation \eqref{eq_general_eq_2}, as functional power series in $G$, given by
    \begin{align} \label{eq_q_expansion}
        \overline{q}_\alpha &= \sum_{t\in[HT^\circ_\circ]} \frac{1}{\sigma(t)} \contraction{S}{G}{t}_\alpha, ~~ \text{and} \\ \label{eq_y_expansion}
        \overline{y}^\alpha &= \sum_{t\in[HT^\circ_\bullet]} \frac{1}{\sigma(t)} \contraction{S}{G}{t}^{\alpha} .
    \end{align}
\end{lemma}
\begin{proof}
Uniqueness is immediate since we are working with formal power series. We first perform the calculation for $\overline{q}_\alpha$, by plugging in equation \eqref{eq_y_expansion}, then using the multivariate Taylor's formula and part (c) of Lemma \ref{lemma_tree_bracket_calc}.
\begin{gather*}
    \overline{q}_\alpha = \dely{\alpha} G(\overline{y}) = \dely{\alpha} G\left(\varphi + \sum_{t\in[HT^\circ_\bullet]^+} \frac{1}{\sigma(t)} \contraction{S}{G}{t}\right) \\
    = \sum_{m \ge 0} \frac{1}{m!} \left( \sum_{t_1\in[HT^\circ_\bullet]^+} \frac{1}{\sigma(t_1)} \contraction{S}{G}{t_1}^{b_1} \right) \dots \left( \sum_{t_m\in[HT^\circ_\bullet]^+} \frac{1}{\sigma(t_m)} \contraction{S}{G}{t_m}^{b_m} \right)  \dely{\alpha}\dely{b_1}\dots\dely{b_m} G(\varphi) \\
    = \sum_{m \ge 0} \sum_{t_1\in[HT^\circ_\bullet]^+} \dots \sum_{t_m\in[HT^\circ_\bullet]^+} \frac{1}{m!\sigma(t)} (\mu_1! \mu_2! \dots) \contraction{S}{G}{t}_\alpha \\
    = \sum_{t\in[HT^\circ_\circ]} \frac{1}{\sigma(t)} \contraction{S}{G}{t}_\alpha, ~~ \text{for} ~~ t = [t_1,\dots,t_m]_{\multimapinv} .
\end{gather*}
The previous calculation is the reason we only consider trees with white leaves, in contrast to \cite{operad}. Otherwise, the final equality would not hold. 

Next we perform the calculation for $\overline{y}^\alpha$, by plugging in equation \eqref{eq_q_expansion} then using Leibniz's rule and Lemma \ref{lemma_tree_bracket_calc} part (d), where we assume that the coefficients $S^{a_1 \dots a_i}$ of $S$ are symmetric in the indices.
\begin{gather*}
    \overline{y}^\alpha = \frac{\partial}{\partial q^\alpha} S(\overline{q}) = \sum_{m \ge 0} (m+1)  S^{\alpha b_1 \dots b_m} \overline{q}_{a_1} \dots \overline{q}_{a_m} \\
    = \sum_{m \ge 0} \sum_{t_1\in[HT^\circ_\circ]} \dots \sum_{t_m\in[HT^\circ_\circ]} \frac{1}{m!\sigma(t)} (\mu_1! \mu_2! \dots) (m+1)! S^{\alpha b_1 \dots b_m} \contraction{S}{G}{t_1}_{b_1} \dots \contraction{S}{G}{t_m}_{b_m}  \\
    = \sum_{t\in[HT^\circ_\bullet]} \frac{1}{\sigma(t)} \contraction{S}{G}{t}^{\alpha}, ~~ \text{for} ~~ t = [t_1,\dots,t_m]_\multimapdotinv
\end{gather*}
\end{proof}

Next we find an expansion over rooted trees for $R(S,G)$ .

\begin{lemma}\label{lemma_H_calc_1}
    We have a unique solution $R(S,G)$ to equation \eqref{eq_general_eq_1} given by
    \begin{gather*}
        R(S,G) = \sum_{t \in [RT^\circ]} \frac{1}{\sigma(t)} \contraction{S}{G}{t} \\ - \left( \sum_{t\in[HT^\circ_\bullet]} \frac{1}{\sigma(t)} \contraction{S}{G}{t}^{\alpha} \right) \left( \sum_{t\in[HT^\circ_\circ]} \frac{1}{\sigma(t)} \contraction{S}{G}{t}_\alpha \right)
    \end{gather*}
\end{lemma}
\begin{proof}
To prove the lemma we calculate expansions for each term in equation \eqref{eq_general_eq_1} in turn. We first calculate $G(\overline{y})$ by plugging in the solution to $\overline{y}$ from equation \eqref{eq_y_expansion} then using Taylor's theorem and Lemma \ref{lemma_tree_bracket_calc} part (a). 
\begin{gather*}
    G(\overline{y}) =  G\left(\varphi + \sum_{t\in[RT^\circ_\bullet]^+} \frac{1}{\sigma(t)} \contraction{S}{G}{t}\right) \\
    = \sum_{m \ge 0} \frac{1}{m!} \left( \sum_{t_1\in[RT^\circ_\bullet]^+} \frac{1}{\sigma(t_1)} \contraction{S}{G}{t_1}^{b_1} \right) \dots \left( \sum_{t_m\in[RT^\circ_\bullet]^+} \frac{1}{\sigma(t_m)} \contraction{S}{G}{t_m}^{b_m} \right)  \dely{b_1}\dots\dely{b_m} G(\varphi) \\
    = \sum_{m \ge 0} \sum_{t_1\in[RT^\circ_\bullet]^+} \dots \sum_{t_m\in[RT^\circ_\bullet]^+} \frac{1}{m!\sigma(t)} (\mu_1! \mu_2! \dots) \contraction{S}{G}{t} \\
    = \sum_{t\in[RT^\circ_\circ]} \frac{1}{\sigma(t)} \contraction{S}{G}{t}, ~~ \text{for} ~~ t = [t_1,\dots,t_m]_{\circ} .
\end{gather*}

Next, we calculate $S(\overline{q})$ by plugging in equation \eqref{eq_q_expansion}, expanding the resulting product, and using Lemma \ref{lemma_tree_bracket_calc} part (b).
\begin{gather*}
    S(\overline{q}) = \sum_{m \ge 0} S^{b_1 \dots b_m} \overline{q}_{a_1} \dots \overline{q}_{a_m}  \\
    = \sum_{m \ge 0} \sum_{t_1\in[RT^\circ_\circ]} \dots \sum_{t_m\in[RT^\circ_\circ]} \frac{1}{m!\sigma(t)} (\mu_1 \mu_2! \dots) m! S^{b_1 \dots b_m} \contraction{S}{G}{t_1}_{b_1} \dots \contraction{S}{G}{t_m}_{b_m}  \\
    = \sum_{t\in[RT^\circ_\bullet]} \frac{1}{\sigma(t)} \contraction{S}{G}{t}, ~~ \text{for} ~~ t = [t_1,\dots,t_m]_\bullet
\end{gather*}

Finally, we note that $[RT^\circ_\circ] \cup [RT^\circ_\bullet] = [RT^\circ]$ and plug in equations \eqref{eq_q_expansion} and \eqref{eq_y_expansion} into the term $\overline{y}^\alpha \overline{q}_\alpha$.

\end{proof}

All that is left to complete the proof of Proposition \ref{prop_general_eq_graph} is to clean up the above expansion and write it in terms of unrooted trees. This requires a couple of lemmas and a definition which we present first before completing the proof. 

\begin{definition}[Butcher Product]
    We define the Butcher product $\tau \circ \theta \in [RT_\bullet]$ of two hooked trees $\tau \in [HT_\bullet]$ and $\theta \in [HT_\circ]$ by recursion on $\tau$. If $u = \multimapdotinv$, then define $\tau \circ \theta = [\theta]_\bullet$. If $\tau = [\tau_1, \dots, \tau_m]_\multimapdotinv$, then define $\tau \circ \theta = [\tau_1, \dots, \tau_m, \theta]_\bullet$. Note that the same procedure works if the root colours of $\tau$ and $\theta$ are swapped. The procedure is the same as attaching the half-connected edges of $\tau$ and $\theta$ and declaring the root of $\tau$ to be the root of $\tau \circ \theta$.
\end{definition}

\begin{lemma}\label{lemma_eldiff_circ}
    For any $\tau \in [HT_\bullet]$ and $\theta \in [HT_\circ]$ we have 
    \begin{equation*}
        \contraction{S}{G}{\tau}^{\alpha} \contraction{S}{G}{\theta}_\alpha = \contraction{S}{G}{\tau \circ \theta} .
    \end{equation*}
\end{lemma}
\begin{proof}
    Proceed by induction on $\tau$. If $\tau = \multimapdotinv$, then $\tau \circ \theta = [\theta]_\bullet$ so
    \begin{align*}
        \contraction{S}{G}{\tau \circ \theta} & = \varphi^{\alpha} \contraction{S}{G}{\theta}_\alpha \\ 
        & = \contraction{S}{G}{\tau}^{\alpha} \contraction{S}{G}{\theta}_\alpha  .
    \end{align*}
    If $\tau = [\tau_1, \dots, \tau_m]_\multimapdotinv$, then 
    \begin{align*}
        \contraction{S}{G}{\tau}^{\alpha} \contraction{S}{G}{\theta}_\alpha & = (m+1)! S^{\alpha b_1 \dots b_m} \contraction{S}{G}{\tau_1}_{b_1} \dots \contraction{S}{G}{\tau_m}_{b_m}  \contraction{S}{G}{\theta}_\alpha \\
        & = \contraction{S}{G}{\tau \circ \theta},
    \end{align*}
    as required.
\end{proof}

\begin{lemma}\label{lemma_sym_frac}
    Let $t = (B_t,W_t,E_t) \in T$. Given a vertex of any colour $v \in V_t$, denote by $t_v \in RT$ the rooted bipartite tree with chosen root $v$. Then, for any $v \in V_t$ we have
    \begin{equation*}
        \frac{|sym(t)|}{|sym(t_v)|} = |\{ v' \in V_t | t_{v'} ~\text{is isomorphic (as a rooted tree) to}~ t_v\}| =: k(t,v), 
    \end{equation*}
    and for any $e = (u,v) \in E_t$ we have
    \begin{equation*}
        \frac{|sym(t)|}{|sym(t_u)||sym(t_v)|} = |\{ e'=(u', v') \in E_t | t_{u'} \sqcup t_{v'} ~\text{is isomorphic to}~ t_u \sqcup t_v\ \}| =: l(t,e).  
    \end{equation*}
\end{lemma}
\begin{proof}
    See \cite{operad} Lemma 4, pg 24. 
\end{proof}

\begin{proof}[Proof of Proposition \ref{prop_general_eq_graph}]

By Lemma \ref{lemma_H_calc_1} and Lemma \ref{lemma_eldiff_circ} we have
\begin{align*}
    R(S,G) & = \sum_{t \in [RT^\circ]} \frac{1}{\sigma(t)} \contraction{S}{G}{t} - \sum_{\tau\in[RT^\circ_\bullet]} \sum_{\theta\in[RT^\circ_\circ]} \frac{1}{\sigma(\tau) \sigma(\theta)}\contraction{S}{G}{\tau}^{\alpha} \contraction{S}{G}{\theta}_\alpha \\
    & = \sum_{t \in [RT^\circ]} \contraction{S}{G}{t} \left( \frac{1}{\sigma(t)}  - \sum_{\substack{\tau\in[RT^\circ_\bullet] , ~ \theta\in[RT^\circ_\circ] \\ \tau \circ \theta = t}}  \frac{1}{\sigma(\tau)\sigma(\theta)} \right) .
\end{align*}

Now note that when $\tau = \bullet$, the Butcher product $\tau \circ \theta$ is just $[\theta]_\bullet$ and $\sigma(t) = \sigma(\tau \circ \theta) = \sigma([\theta]_\bullet) = \sigma(\theta) = \sigma(\bullet)\sigma(\theta)$. Therefore, if the root of $t$ is black and has degree one, so that $t = \bullet \circ \theta$ for a unique $\theta\in[RT^\circ_\circ]$, the resulting term is zero. We have reduced our sum to be over the set $[RB] := [RT^\circ] - \{ \bullet \} \circ [RT^\circ_\circ]$. Crucially, the set $[B]$ of unrooted trees in $[RB]$ is equal to $[T^\circ]$. So get 
\begin{align*}
    R(S,G) & = \sum_{t \in [RB]} \contraction{S}{G}{t} \left( \frac{1}{\sigma(t)}  - \sum_{\substack{\tau\in[RT^\circ_\bullet] , ~ \theta\in[RT^\circ_\circ] \\ \tau \circ \theta = t}}  \frac{1}{\sigma(\tau)\sigma(\theta)} \right) \\ 
    & = \sum_{t \in [B]} \contraction{S}{G}{t} \left(  \sum_{v \in V_t} \frac{1}{\sigma(t_v)} \frac{1}{k(t,v)} -  \sum_{e = (u,v) \in E_t} \frac{1}{\sigma(t_u)\sigma(t_v)} \frac{1}{l(t,e)}   \right) \\
    & = \sum_{t \in B} \frac{1}{|t|_\circ! |t|_\bullet!} \contraction{S}{G}{t} \left(  \sum_{v \in V_t} \frac{|sym(t)|}{|sym(t_v)|} \frac{1}{k(t,v)} -  \sum_{e = (u,v) \in E_t} \frac{|sym(t)|}{|sym(u)||sym(v)|} \frac{1}{l(t,e)}   \right) \\
    & = \sum_{t \in T^\circ} \frac{1}{|t|_\circ! |t|_\bullet!} \contraction{S}{G}{t} ,
\end{align*}
as required. For the third equality, we used that the number of trees isomorphic to $t$ is equal to $(|t|_\circ! |t|_\bullet!)/|sym(t)|$. The last equality follows from Lemma \ref{lemma_sym_frac} and that the number of edges in a tree is always one less than the number of vertices.

\end{proof}

We will now apply Proposition \ref{prop_general_eq_graph} to the three cases of interest, starting with pullbacks.

\subsection{Graphical Calculus for Pullbacks by Thick Morphisms}\label{sec_pullback_graph}

Consider the pullback of $g\in C^\infty(M_2)$ by $\Phi : M_1 \text{\ding{225}} M_2$, as defined by equations \eqref{eq_pullback_1} and \eqref{eq_pullback_2}. To apply Proposition \ref{prop_general_eq_graph} to pullbacks we need only replace $G$ with $g$ and reintroduce the dependence of the coefficients of $S$ on $x$, which makes no difference to the formulas. For clarity we present the specific formulas that result. 

\begin{definition}
    Given a generating function $S(x,q) = S_0(x) + \varphi^i(x)q_i + S^{ij}(x)q_iq_j + \dots$, a smooth function $g\in C^\infty(M_2)$, and a bipartite tree $t = (B_t, W_t, E_t)$, we define
    \begin{equation}\label{eq_pullback_contraction_def}
        \contraction{S}{g}{t}(x) := \left( \prod_{v \in B_t} m_v! S^{e^v_1, \dots, e^v_{m_v}}(x) \right) \left( \prod_{w \in W_t} \dely{e^w_1} \dots \dely{e^w_{m_w}} g(\varphi(x)) \right),
    \end{equation}
    where for each $v \in V_t$, we denote the elements of $E_v$, the set of edges connected to $v$, by $e^v_1, \dots, e^v_{m_v}$, and the labels $e^v_k$ are used as dummy indices to be summed over by the Einstein summation convention.  
\end{definition}

Then by Proposition \ref{prop_general_eq_graph}, the pullback of $g$ by $\Phi$ is given by  
\begin{equation}\label{eq_pullback_expansion}
    \Phi^*[g](x) = \sum_{t \in T^\circ} \frac{1}{|t|_\circ ! |t|_\bullet !} \contraction{S}{g}{t}(x) = \sum_{t \in [T^\circ]} \frac{1}{|sym(t)|} \contraction{S}{g}{t}(x) .
\end{equation}
Note that the order of $g$ in $\contraction{S}{g}{t}$ is given by the number of white vertices $|t|_\circ$.

\subsection{Graphical Calculus for Compositions of Thick Morphisms}\label{sec_comp_graph}

Consider the composition of thick morphisms $\Phi_H = \Phi_G \circ \Phi_F: M_1 \text{\ding{225}} M_3$ defined by equations \eqref{eq_thick_comp_1} and \eqref{eq_thick_comp_2}. To apply Proposition \ref{prop_general_eq_graph} we simply exchange $F(x,q)$ for $S(q)$ again simply introducing the dependence of the coefficients of $F$ on $x$ and $G$ on the momentum coordinates $r$ on $M_3$. Now by Proposition \ref{prop_general_eq_graph} we get an expansion almost identical to that for pullbacks. However since $G(y,r)$ is a formal power series in $r$, we'd like to expand over the terms in $G$, thereby obtaining an expansion of the generating function $H(x,r)$ as a formal power series in $r$. This is achieved by adding "weights" to the white-vertices of the trees.

\begin{definition}
   A white-weighted bipartite tree is a bipartite tree $t = (B_t, W_t, E_t)$ with a map $L_t: W_t \rightarrow \mathbb{Z}_{\ge 0}$. An isomorphism of white-weighted bipartite trees is an isomorphism of bipartite trees $(\sigma_\bullet,\sigma_\circ): t \xrightarrow{\cong} t'$ such that $L_{t'}(\sigma_\circ(v)) = L_t(v)$ for all $v \in W_t$. For any set of trees $A$, we use brackets $[A]$ to denote the trees up to isomorphism, and the subscript $A_{\infty/2}$ to denote the white-weighted versions of trees in $A$. Following \cite{operad} we abuse notation and write $[A]_{\infty/2}$ instead of $[A_{\infty/2}]$. We define $\|t\|_\circ$ to be the sum of the weights of all the white vertices of t. 
\end{definition}

We now define a multi-index $I_t$ for each white-weighted bipartite tree $t$. This keeps track of the indices of $G$ that contract with the momenta $r$ in the expansion.
\begin{definition}[The multi-index $I_t$]\label{def_multi_index_I}
    For any white-weighted bipartite tree $t = (B_t, W_t, E_t, L_t)$, we write $\ell_{w} = L_t(w)$ for the weights associated to each vertex $w \in W_t$ and define
    \begin{equation*}
        I_t := \prod_{w \in W_t} a_{w, 1} \dots a_{w, \ell_w} .
    \end{equation*}
\end{definition} 
\noindent Each $a_{w,k}$ is just a label whose purpose is to be contracted with a momentum factor $r_{a_{w,k}}$ according to the Einstein summation convention. Note that the length of $I_t$ is just $\|t\|_\circ$. 

We can now write down the formula for terms specific to the composition of thick morphisms.
\begin{definition}
    Given generating functions $F(x, q) = F_0(x) + \varphi^i(x)q_i + F^{ij}(x)q_iq_j + \dots$ and
    $G(y,r) = G_0(y) + \gamma^i(y)r_i + G^{ij}(y)r_ir_j + \dots$, and a white weighted bipartite tree $t = (B_t, W_t, E_t, L_t)$, we define
    \begin{equation}\label{eq_comp_contraction_def}
        \contraction{F}{G}{t}^{I_t}(x) := \left( \prod_{v \in B_t} m_v! F^{e^v_1, \dots, e^v_{m_v}(x)} \right) \left( \prod_{w \in W_t} \dely{e^w_1} \dots \dely{e^w_{m_w}} G^{a_{w, 1} \dots a_{w, \ell_w}}(\varphi(x)) \right),
    \end{equation}
    where for each $v \in V_t$, we denote the elements of $E_v$, the set of edges connected to $v$, by $e^v_1, \dots, e^v_{m_v}$, and the labels $e^v_k$ are used as dummy indices to be summed over by the Einstein summation convention.  
\end{definition}

Then by Proposition \ref{prop_general_eq_graph} and the linearity of partial derivatives, we get the following expansion for $H(x,r)$ as a formal power series in $r$
\begin{equation}\label{eq_comp_expansion}
        H(x,r) = \sum_{t \in T_{\infty/2}^\circ} \frac{r_{I_t}}{|t|_\circ! |t|_\bullet!} \contraction{F}{G}{t}^{I_t}(x) = \sum_{t \in [T^\circ]_{\infty/2}} \frac{r_{I_t}}{|sym(t)|} \contraction{F}{G}{t}^{I_t}(x),
    \end{equation}
where we have used the shorthand $r_I = r_{i_1} \dots r_{i_n}$ for any multi-index $I = (i_1,\dots,i_n)$.

\subsection{Graphical Calculus for Coordinate Transformations of Thick Morphisms}\label{sec_transformation_graph}

Consider an invertible change of local coordinates on $M_1\times M_2$ given by $ $$x = x(x')$, $y = y(y')$, with $p', q'$ the corresponding conjugate momenta, under which the generating function for a thick morphism $S(x,q)$ in the "old" coordinates transforms by equations \eqref{eq_thick_transformation_1} and \eqref{eq_thick_transformation_2}. To apply Proposition \ref{prop_general_eq_graph} we first replace $G(y)$ with $y^{i'}(y)q_{i'}$, then reintroduce the dependence of the coefficients of $S$ on $x$, and use the substitution $x = x(x')$. To obtain the expansion as a power series in $q'$, we also pull the $q'$ factors outside of the derivatives by $y$, keeping track of which factors contract with which $y'$ factors by introducing a simpler multi-index $\alpha_t$ for any bipartite tree.

\begin{definition}[The multi-index $\alpha_t$]\label{def_multi_index_alpha}
    For bipartite tree $t = (B_t, W_t, E_t)$, we define
    \begin{equation*}
        \alpha_t := \prod_{w \in W_t} a_w .
    \end{equation*}
\end{definition}
\noindent Each $a_w$ is just a label whose purpose is to be contracted with a momentum factor $q'_{a_w}$ according to the Einstein summation convention. Note that the length of $\alpha_t$ is just $|t|_\circ$.

We can now write down the formula for terms resulting from coordinate transformations.
\begin{definition}
    Given a generating function $S(x,q) = S_0(x) + \varphi^i(x)q_i + S^{ij}(x)q_iq_j + \dots$, a change of coordinates $x = x(x')$, $y = y(y')$, and a bipartite tree $t = (B_t, W_t, E_t)$, we define
    \begin{equation}\label{eq_transformation_contraction_def}
        \contraction{S}{y'}{t}^{\alpha_t}(x') := \left( \prod_{v \in B_t} m_v! S^{e^v_1, \dots, e^v_{m_v}}(x(x')) \right) \left( \prod_{w \in W_t} \dely{e^w_1} \dots \dely{e^w_{m_w}} y^{a_w'}(\varphi(x(x'))) \right),
    \end{equation}
    where for each $v \in V_t$, we denote the elements of $E_v$, the set of edges connected to $v$, by $e^v_1, \dots, e^v_{m_v}$, and the labels $e^v_k$ are used as dummy indices to be summed over by the Einstein summation convention.  
\end{definition}

Then by Proposition \ref{prop_general_eq_graph}, and by pulling out all $q'$ factors from any partial derivatives with respect to $y$, noting that there will be one factor of $q'$ for each white vertex in a tree, we obtain the expansion for the generating function $S'(x',q')$ in the "new" coordinates
\begin{equation}
    S'(x',q') = \sum_{t \in T^\circ} \frac{q'_{\alpha_t}}{|t|_\circ! |t|_\bullet!} \contraction{S}{y'}{t}^{\alpha_t}(x') = \sum_{t \in [T^\circ]} \frac{q'_{\alpha_t}}{|sym(t)|} \contraction{S}{y'}{t}^{\alpha_t}(x') .
\end{equation}

\section{The Graphical Calculus for Quantum Thick Morphisms}\label{sec_quantum_graphical_calc}

We would now like to extend the graphical calculus to quantum thick morphisms. In \cite{Voronov_microformal_homotopyalg} Theorem 10, Voronov showed that classical thick morphisms are in some sense the classical limit of quantum thick morphisms (taking $\hbar \rightarrow 0$). So we might expect the graphical calculus in the classical case to be a classical limit of the quantum case, and this is indeed what we find. In the classical case we sum over trees, that is graphs with no cycles. In the quantum case we sum over all graphs, with cycles contributing factors of $\hbar$. 

\subsection{Preliminaries on Quantum Thick Morphisms}\label{sec_qprelim}

\begin{definition}
    Let $M_1$ and $M_2$ be supermanifolds. A quantum generating function $S_\hbar(x,q)$, associated to canonical coordinate systems $x^a, p_a$ and $y^i, q_i$ on $T^*M_1$ and $T^*M_2$, is a formal power series in $q$ with coefficients that are formal power series in $\hbar$,
    \begin{equation*}
        S_\hbar(x,q) = S^0_\hbar(x) + \varphi^i_\hbar(x)q_i + S^{ij}_\hbar(x)q_jq_i + S^{ijk}_\hbar(x)q_kq_jq_i + \dots .
    \end{equation*}
    A quantum thick morphism $\Phi : M_1 \text{\ding{225}}_q M_2$ is given by a collection of quantum generating functions in each coordinate system, with the transformation law given below. Quantum thick morphisms are identified with their pullback action on functions, which we also give below.
\end{definition}

Again we constrain ourselves temporarily to the setting of regular (non-super) manifolds for simplicity. 

We now recall the formulae governing pullbacks, compositions, and changes of coordinates for quantum thick morphisms. In each case let $M_1, M_2, M_3$ be manifolds, and let $\Phi : M_1 \text{\ding{225}}_q M_2$, $\Phi_F: M_1 \text{\ding{225}}_q M_2$ and $\Phi_G: M_2 \text{\ding{225}}_q M_3$ be quantum thick morphisms between them. Choosing suitable local coordinates $x, y, z$ on $M_1, M_2, M_3$ respectively with corresponding conjugate momenta $p, q, r$, we have the following generating functions for $\Phi$, $\Phi_F$ and $\Phi_G$ given as formal power series in the momenta.
\begin{align*}
    S_{\hbar}(x,q) & = S^0_{\hbar}(x) + \varphi^i_{\hbar}(x)q_i + S^{ij}_{\hbar}(x)q_iq_j + \dots \\ 
    F_{\hbar}(x, q) & = F^0_{\hbar}(x) + \varphi^i_{\hbar}(x)q_i + F^{ij}_{\hbar}(x)q_iq_j + \dots \\
    G_{\hbar}(y,r) & = G^0_{\hbar}(y) + \gamma^i_{\hbar}(y)r_i + G^{ij}_{\hbar}(y)r_ir_j + \dots .
\end{align*}
The subscripts of $\hbar$ indicate that these terms are formal power series in $\hbar$. Given these generating functions we then have the following equations.
\begin{q_pullback_law}
    For any oscillatory wave function $w \in OC^\infty_\hbar(M_2) := C^\infty(M_2)[[\hbar]] \exp (\frac{i}{\hbar} C^\infty(M_2))$, the pullback $(\hat{\Phi}^* w)(x)$ is given by
    \begin{equation}\label{eq_qpullback}
        (\hat{\Phi}^* w)(x) = \int_{T^*M_2} Dy \text{\textit{\DJ}} q \, \expih{(S_{\hbar}(x,q) - y^\alpha q_\alpha)} w(y) ,
    \end{equation}
    where $\text{\textit{\DJ}} q := (2\pi \hbar)^{-d_0}(i\hbar)^{d_1}D q$, and $d_0|d_1$ is the dimension of $M_2$. 
\end{q_pullback_law}

\begin{q_composition_law}
    The composition $\hat{\Phi}_H = \hat{\Phi}_G \circ \hat{\Phi}_F: M_1 \text{\ding{225}}_q M_3$ is a quantum thick morphism and has generating function $H_{\hbar}(x,r)$ given by
    \begin{equation}\label{eq_qthick_comp}
        \expih{H_{\hbar}(x,r)} = \int_{T^*M_2} Dy \text{\textit{\DJ}} q \, \expih{(G_{\hbar}(y,r) + F_{\hbar}(x,q) - y^\alpha q_\alpha)}.
    \end{equation}
\end{q_composition_law}

\begin{q_transformation_law}
    Let $x = x(x')$, $y = y(y')$ be an invertible change of local coordinates, with $p', q'$ the corresponding conjugate momenta. If $S(x,q)$ is the generating function for $\Phi$ in the "old" coordinate system $(x,p,y,q)$, then the generating function $S'(x',q')$ in the "new" coordinate system $(x',p',y',q')$ satisfies 
    \begin{equation}\label{eq_qthick_transformation}
         \expih{S'_{\hbar}(x',q')} = \int_{T^*M_2} Dy \text{\textit{\DJ}} q \, \expih{(S_{\hbar}(x(x'),q) + y^{\alpha'}(y)q_{\alpha'} - y^\alpha q_\alpha)}.
    \end{equation}
\end{q_transformation_law}

For pullbacks we now restrict ourselves to the case $w(y) = \expih{g(y)}$. Then all three equations share the same form, so we again write down a general equation capturing this form, which we will again solve using a graphical calculus described in section \ref{sec_qgen_eq_solve}.

\begin{q_general_eq}
     Consider a formal power series $S(q) = S_0 + \varphi^iq_i + S^{ij}q_iq_j+ \dots$, and a smooth function $G(y) : \mathbb{R}^{d} \rightarrow \mathbb{R}$. The coefficients of $S$ and the function $G$ may depend on other parameters which we ignore to simplify notation (in our cases the coefficients of $S$ depend on $x$ and $G$ is a formal power series in either $r$ or $q_{\alpha'}$, and both $S$ and $G$ are formal power series in $\hbar$). We are interested in the result of the following integral
        \begin{equation}\label{eq_qgeneral_eq}
            I(S,G) =  \int Dy \text{\textit{\DJ}} q \, \expih{(S(q) + G(y) - y^\alpha q_\alpha)},
        \end{equation}
        where the integration is over $\mathbb{R}^d$ and its dual. 
\end{q_general_eq}

\subsection{Solving the Quantum General Equation}\label{sec_qgen_eq_solve}

We first rewrite the integral in equation \eqref{eq_qgeneral_eq} as the application of a formal $\hbar$-differential operator (for details on formal $\hbar$-differential operatos see \cite{SHEMYAKOVA_BV}) evaluated at the point $y^i = \varphi^i$ defined by the first order coefficients in $S$. This is the generalisation of \cite{Voronov_microformal_homotopyalg} Theorem 11 to the Quantum General Equation.
\begin{lemma}\label{lemma_qgen_eq_exp}
    The integral $I(S,G)$ can be expressed as
    \begin{equation}\label{eq_qgen_eq_exp}
        I(S,G) = \expih{S^0} \left( \expih{S^+(\frac{\hbar}{i}\dely{}) } \expih{G(y)} \right)_{\Big\arrowvert{y^i = \varphi^i}} ,
    \end{equation}
    where $S^+(q)$ is the sum of all the terms of greater than first order in $q$ in $S(q)$.
\end{lemma}
\begin{proof}
    The proof is the same as for \cite{Voronov_microformal_homotopyalg} Theorem 11. We have 
        \begin{equation*}
    S(q) = S^0 + \varphi^i q_i + S^+(q),
    \end{equation*}
    which we substitute into equation \eqref{eq_qgeneral_eq} to get
    \begin{align*}
        I(S,G) & = \int Dy \text{\textit{\DJ}} q \, \expih{(S^0 + \varphi^i q_i + S^+(q) + G(y) - y^\alpha q_\alpha)} \\
        & = \expih{S^0} \int \text{\textit{\DJ}} q \, \expih{\varphi^i q_i} \expih{S^+(q)} \int Dy \, e^{ - \frac{i}{\hbar} y^\alpha q_\alpha} \expih{G(y)} .
    \end{align*}
     The $Dy$ integral is the Fourier transform of the function $\expih{G(y)}$ from the variables $y^i$ to the variables $q_i$. This Fourier transform is then multiplied by $\expih{S^+(q)}$, which is just a function of $q$. The $\text{\textit{\DJ}} q$ integral performs the inverse Fourier transform from $q_i$ back to $y^i$ but with $\varphi^i$ substituted for $y^i$. Under this inverse Fourier transform, the multiplications by $q_i$ terms in $\expih{S^+(q)}$ become partial derivatives with respect to $y^i$ and the result follows.  
\end{proof}

We now see that all terms inside the brackets of equation \eqref{eq_qgen_eq_exp} (ignoring $\hbar$ and $i$ factors for now) consist of collections of differential operators of the form $S^{a_1 \dots a_m} \dely{a_1} \dots \dely{a_m}$, with $m \ge 2$, acting on the exponential of $G(y)$. Using the chain rule, we see that the first derivative brings down a factor of $\dely{\alpha}G$, for some index $\alpha$ which is contracted with and index on a term of $S$. Then by the product rule, the further derivatives split into sums, and in each summand the derivative can either bring down extra factors of $\dely{\beta}G$, where $\beta$ is another contracted index, or add derivatives $\dely{\beta}$ to existing "downstairs" factors of $G$. 

The resulting terms can be represented by an extension of the Feynman rules for classical thick morphisms to all bipartite graphs, including disconnected graphs, and not just trees. For example, the same term of $S$ can contract with derivatives acting on the same factor from $G$, resulting in a loop. Turning this observation into a proper graphical calculus involves a fairly simple combinatorial argument which is laid out in the following lemma and proposition. 

We denote by $\Gamma$ the set of all connected bipartite graphs, by $\overline{\Gamma}$ the set of all bipartite graphs (includes the disconnected graphs). 
\begin{definition}
    To each bipartite graph $\gamma = (B_\gamma, W_\gamma, E_\gamma)$ we associate a term $\contraction{S}{G}{\gamma}$ by the following. For each vertex of any colour $v \in V_\gamma$, let $E_v = \{ e \in E_\gamma ~ | ~ \text{$e$ is connected to $v$}\}$ and denote the elements of $E_v$ by $e^v_1, \dots, e^v_{m_v}$. We then use the labels $e^v_k$ as dummy indices to be summed over by the Einstein summation convention and define
    \begin{equation}\label{eq_qcontraction_def}
        \contraction{S}{G}{\gamma} := \left( \prod_{v \in B_\gamma} \frac{m_v!}{(1!)^{k^v_1} (2!)^{k^v_2} \dots (m_v!)^{k^v_{m_v}}} S^{e^v_1, \dots, e^v_{m_v}} \right) \left( \prod_{w \in W_\gamma} \dely{e^w_1} \dots \dely{e^w_{m_w}} G(\varphi) \right),
    \end{equation}
    where again $m_v$ is the number of edges connected to the vertex $v$ and for each $i = 1, \dots m_v$, we let $k^v_i$ be the number of white vertices with which $v$ shares $i$ edges. Note that if the graph is a tree, then this reduces to the formula in equation \ref{eq_contraction_def}, since $v$ can connect with any white vertex only once.
\end{definition}
The Feynman-esque rules encoding this formula are the same as those in section \ref{sec_trees_and_Feynman_rules}, except that for each black vertex $v$ in a bipartite graph, we now write a factor 
\begin{equation}\label{eq_Feynman_rules_extension}
    \frac{m_v!}{(1!)^{k^v_1} (2!)^{k^v_2} \dots (m_v!)^{k^v_{m_v}}} S^{* \dots *} ,
\end{equation}
where again the number of upper indices of $S$ is equal to $m_v$. The meaning of this combinatorial factor will become clear in the proof of the next lemma. 

\begin{lemma}\label{lemma_qgen_eq_initial}
    The integral I(S,G) can be written as 
    \begin{equation}\label{eq_qgen_eq_initial}
        I(S,G) = \expih{S^0} \left( 1 + \sum_{\gamma \in [\overline{\Gamma}^\circ] - \{\circ\} } \frac{1}{|sym(\gamma)|}   \left(\frac{\hbar}{i}\right)^{|E_\gamma| - |\gamma|}  \contraction{S}{G}{\gamma}  \right) \expih{G(\varphi)} . 
    \end{equation}
\end{lemma}
\begin{proof}
    We start with equation \eqref{eq_qgen_eq_exp}, noticing that the $\expih{S^0}$ term is unchanged in \eqref{eq_qgen_eq_initial}. First we expand the $\expih{S^+(\frac{\hbar}{i}\dely{})}$ factor to get
    \begin{align}\label{eq_exp_of_S_calc}
        \expih{S^+(\frac{\hbar}{i}\dely{})} & = 1 + \sum_{n = 1}^\infty \frac{1}{n!} \left(\frac{i}{\hbar}\right)^{n} \left( S^+(\frac{\hbar}{i}\dely{}) \right)^n \nonumber\\ 
        & = 1 + \sum_{n = 1}^\infty \frac{1}{n!} \left(\frac{i}{\hbar}\right)^{n} \sum_{(m_1, \dots, m_n) \in \mathbb{Z}_{\ge 2}^n } \left(\frac{\hbar}{i}\right)^{m_1} S^{a_{1,1} \dots a_{1,m_1}} \dely{a_{1,1}} \dots \dely{a_{1,m_1}} \dots \left(\frac{\hbar}{i}\right)^{m_n} S^{a_{n,1} \dots a_{n,m_n}} \dely{a_{n,1}} \dots \dely{a_{n,m_n}} ,
    \end{align}
    where $\mathbb{Z}_{\ge 2}^n$ is just length $n$ ordered lists of integers at least 2. 
    
    At this point it is helpful to visualise bipartite graphs in the following way. Think of all the black vertices as being in a line on the left and the white vertices being in a line on the right. Bipartiteness ensures, and can be defined by, the property that all edges go only from one side to the other, as in the diagram below.
    
    
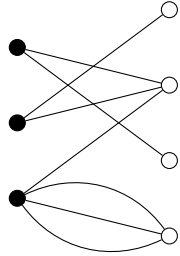
\[\begin{tikzpicture}[x=2cm, y=1cm]
	\vertex[fill] (a) at (0,0) {};
	\vertex[fill] (b) at (0,1) {};
	\vertex[fill] (c) at (0,2) {};
	\vertex (d) at (1,-0.5) {};
        \vertex (e) at (1,0.5) {};
        \vertex (f) at (1,1.5) {};
        \vertex (g) at (1,2.5) {};
	\path
		(a) edge[bend right = 40] (d)
            (a) edge[bend left = 40] (d)
            (a) edge (d)
		(b) edge (g)
		(c) edge (e)
            (b) edge (f)
            (a) edge (f)
            (c) edge (f)
	;
    \end{tikzpicture}\]
    The terms in our expansion of $\expih{S^+(\frac{\hbar}{i}\dely{})}$ can be thought of as all the possible configurations of black vertices with at least two edges coming from them, but where we have not yet decided where to attach these edges. To work this out we need apply the derivatives to $\expih{G(y)}$. 

    Let's calculate a general term of the form
    \begin{equation}\label{eq_calc_genterm}
         \left(\frac{\hbar}{i}\right)^{m_1 + \dots + m_n} S^{a_{1,1} \dots a_{1,m_1}} \dots S^{a_{n,1} \dots a_{n,m_n}} \left( \dely{a_{1,1}} \dots \dely{a_{1,m_1}} \right) \dots \left( \dely{a_{n,1}} \dots \dely{a_{n,m_n}} \right) \expih{G(y)} .
    \end{equation}
    In our graphical perspective this is a line of $n$ black vertices, where the $k$\textsuperscript{th} vertex has $m_k$ unconnected edges coming from it, and this line of black vertices acts using its edges as derivatives on $\expih{G(y)}$. To calculate the multiple derivatives we use the combinatorial and multivariate version of Faà di Bruno's formula as presented by Michael Hardy in \cite{Hardy_combinatorics}. Let $A$ be the set $\{ a_{1,1}, \dots, a_{1,m_1}, \dots, a_{n,1}, \dots, a_{n,m_n} \}$, and $P$ be the set of partitions of $A$. For any partition $\pi \in P$, and any block $B \in \pi$, we write $\dely{(B)} = \prod_{a \in B} \dely{a}$. We then have that
    \begin{equation}\label{eq_fdbcalc}
        \left( \dely{a_{1,1}} \dots \dely{a_{1,m_1}} \right) \dots \left( \dely{a_{n,1}} \dots \dely{a_{n,m_n}} \right) \expih{G(y)} = \sum_{\pi \in P} \left(  \prod_{B \in \pi} \frac{i}{\hbar} \dely{(B)} G(y) \right) \expih{G(y)} .
    \end{equation}
    Going back to the graphical viewpoint, as $A$ is the set of edges coming from our line of $n$ black vertices, our result for the calculation of the general term \eqref{eq_calc_genterm} is a sum over all partitions of the set of (currently unconnected) edges. Moreover, each summand is exactly the result of connecting the edges in each block of the partition to a common white vertex, that is the associated derivatives acting on a common factor of $G$. In this way we associate to each term an isomorphism class of bipartite graphs by the Feynman rules. Now we see that our term \eqref{eq_calc_genterm} becomes a sum over all possible configurations of white vertices attached to the line of black vertices with edges defined by the form of the term, all multiplied by a factor of $\expih{G(y)}$. We can write this as a sum over bipartite graphs with a fixed configuration of black vertices. Then, looking back at our expansion of $\expih{S^+(\frac{\hbar}{i}\dely{})}$, which became one plus a sum over all possible configurations of black vertices with at least two edges, we see that the bracketed term in equation \eqref{eq_qgen_eq_exp} is one plus a double sum over all configurations of black and white vertices, that is exactly a sum over all bipartite graphs with white leaves only but not including the graphs with unconnected white vertices, all multiplied by a factor of $\expih{G(y)}$.
    
    At this point we see that there are a number of non-equal terms in our expansion that are given by equal isomorphism classes of graphs. We now count them. Firstly, recalling that white vertices correspond to blocks of partitions from Faà di Bruno's formula, there are a number of non-equal partitions that give rise to the same way of connecting the edges to white vertices and thus the same bipartite graph. The product of the coefficients in equation \eqref{eq_Feynman_rules_extension} counts the total number of ways of permuting the edges for each black vertex, while discounting the permutations that just swap edges already contained in the same block of a partition (that is connected to the same white vertex). Additionally, there are non-equal configurations of black vertices that can result in isomorphic graphs, the total possibilities being the factorial of the number of black vertices, which cancels with the factor of $1/n!$ in \eqref{eq_exp_of_S_calc}. Finally, note that we have overcounted by the number of ways of rearranging black and white vertices that result in equal terms in our expansion, which directly corresponds to the number of symmetries of the resulting bipartite graph. This explains why we divide by $|sym(\gamma)|$.
    
    All together, taking $y^i = \varphi^i$ we see that we have proven the expansion over bipartite graphs in \eqref{eq_qgen_eq_initial} up to some factors of $\hbar/i$. The power of $E_\gamma$ comes from the sum $m_1 + \dots + m_n$ in \eqref{eq_exp_of_S_calc} being the total number of edges in the resulting graph, i.e. $m_1 + \dots + m_n = E_\gamma$. The power of $-|\gamma|$ comes from the power of $-|\gamma|_\bullet$ from \eqref{eq_exp_of_S_calc} (remembering that $n$ is the number of black vertices of the resulting graph), and the factors of $i/\hbar$ associated to each derivative of $G$ in \eqref{eq_fdbcalc}, each of which corresponds to a white vertex. The result then follows. 
\end{proof} 

We would now like to rearrange the the middle factor in \eqref{eq_qgen_eq_initial} into an exponential, thereby obtaining an expansion over graphs for $I(S,G)$. We will find that the factors of $\hbar/i$ in this expansion correspond to the number of "loops" in the associated graph. So we make the following definition. 
\begin{definition}
    For a bipartite graph $\gamma$, we define a number $b_\gamma := |E_\gamma| - |\gamma| + 1$, which counts the number of loops in $\gamma$.
\end{definition}

We are now ready to state and prove the final expansion.

\begin{proposition}\label{prop_qgen_eq_graph}
    The integral I(S,G) has the following expansion over bipartite graphs.  
    \begin{equation}\label{eq_qgen_eq_graph}
        I(S,G) = \exp \left( \frac{i}{\hbar} \sum_{\gamma \in [\Gamma^\circ]} \frac{1}{|sym(\gamma)|}  \left(\frac{\hbar}{i}\right)^{b_\gamma}  \contraction{S}{G}{\gamma} \right)
    \end{equation}
\end{proposition}
\begin{proof}
    The main idea of this proof is a standard trick used in quantum field theory. First we calculate 
    \begin{gather*}
        \exp \left( \sum_{\gamma \in [\Gamma^\circ]^+ - \{\circ\} } \frac{1}{|sym(\gamma)|}  \left(\frac{\hbar}{i}\right)^{|E_\gamma| - |\gamma|}  \contraction{S}{G}{\gamma} \right) \\ = 
        1 + \sum_{n = 1}^\infty \frac{1}{n!} \left[ \left( \sum_{\gamma_1 \in [\Gamma^\circ]^+ - \{\circ\} } \frac{1}{|sym(\gamma_1)|}  \left(\frac{\hbar}{i}\right)^{|E_{\gamma_1}| - |\gamma_1|}  \contraction{S}{G}{\gamma_1} \right) \dots \left( \sum_{\gamma_n \in [\Gamma^\circ]^+ - \{\circ\} } \frac{1}{|sym(\gamma_n)|}  \left(\frac{\hbar}{i}\right)^{|E_{\gamma_n}| - |\gamma_n|}  \contraction{S}{G}{\gamma_n} \right) \right] \\ 
        = 1 + \sum_{n = 1}^\infty \sum_{\gamma_1 \in [\Gamma^\circ]^+ - \{\circ\} } \dots \sum_{\gamma_n \in [\Gamma^\circ]^+ - \{\circ\} } \left[\vphantom{\frac{\frac{f}{g}a}{b}} \frac{1}{n! |sym(\gamma_1)| \dots |sym(\gamma_n)|} \left( \frac{\hbar}{i} \right)^{|E_{\gamma_1}| - |\gamma_1| + \dots + |E_{\gamma_n}| - |\gamma_n|} \times \right. \\ \left. \times ~ \contraction{S}{G}{\gamma_1} \dots \contraction{S}{G}{\gamma_n} \vphantom{\frac{\frac{f}{g}a}{b}}\right] \\
        = 1 + \sum_{\gamma \in [\overline{\Gamma}^\circ]^+ - \{\circ\} } \frac{1}{|sym(\gamma)|}   \left(\frac{\hbar}{i}\right)^{|E_\gamma| - |\gamma|}  \contraction{S}{G}{\gamma} , 
    \end{gather*}
    where $\gamma = \gamma_1 \sqcup \dots \sqcup \gamma_n$ is the disjoint union of the graphs $\gamma_1, \dots, \gamma_n$, and we have used the fact that $|sym(\gamma_1 \sqcup \dots \sqcup \gamma_n)| = (n_1! n_2! \dots)|sym(\gamma_1)| \dots |sym(\gamma_n)| $, where the $n_i$ are the numbers of equal isomorphism classes of graphs in $\gamma_1, \dots, \gamma_n$, and the number of ways of forming $\gamma$ as a disjoint union is given by $n!/((n_1! n_2! \dots))$. Notice that this is the same as the middle factor in \eqref{eq_qgen_eq_initial}, so we get
    \begin{align*}
         I & = \expih{S^0} \exp \left( \sum_{\gamma \in [\Gamma^{\circ}]^{+} - \{\circ\} } \frac{1}{|sym(\gamma)|}  \left(\frac{\hbar}{i}\right)^{|E_\gamma| - |\gamma|}  \contraction{S}{G}{\gamma} \right) \expih{G(\varphi)} \\
         & = \exp \left( \frac{i}{\hbar} \left( S^0 + G(\varphi) + \sum_{\gamma \in [\Gamma^{\circ}]^{+} - \{\circ\} } \frac{1}{|sym(\gamma)|}  \left(\frac{\hbar}{i}\right)^{|E_\gamma| - |\gamma| + 1}  \contraction{S}{G}{\gamma} \right) \right).
    \end{align*}
    From this we get the desired result since $S^0$ is exactly the term associated to the graph with single black vertex and $G(\varphi)$ is the term associated to the single white vertex graph.
\end{proof}

\begin{remark}\normalfont
    In our use cases, the factors of the form $S^{a_1 \dots a_n}$ and $G$ are all themselves formal power series in $\hbar$. If we want a full expansion of $I(S,G)$ as a power series in $\hbar$ we can simply sum over all the coefficients in these factors, keeping track of powers of $\hbar$. Graphically this amounts to adding a new weighting to both the black and white vertices in the graphs.  
\end{remark}

We will now apply Proposition \ref{prop_qgen_eq_graph} to pullbacks, compositions, and transformations. The process in each case is essentially the same as for the classical graphical calculi. In particular, for each case the quantum Feynman rules are the same as those for the classical case, with the extension to graphs with loops given by equation \eqref{eq_Feynman_rules_extension}.

\subsection{Graphical Calculus for Pullbacks by Quantum Thick Morphisms}\label{sec_qpullback_graph}

Consider the pullback of an oscillatory wave function of the form $\expih{g(y)} \in OC^\infty_\hbar(M_2)$ by $\Phi : M_1 \text{\ding{225}} M_2$, as defined by equation \eqref{eq_qpullback}. To apply Proposition \ref{prop_general_eq_graph} to pullbacks we need only replace $G$ with $g$ and reintroduce the dependence of the coefficients of $S$ on $x$ and $\hbar$, which makes no difference to the formulas. We now extend the formula in equation \eqref{eq_pullback_contraction_def} for the terms in the pullback to graphs with loops. 
\begin{definition}
    Given a generating function $S_\hbar(x,q) = S^0_\hbar(x) + \varphi_\hbar^i(x)q_i + S_\hbar^{ij}(x)q_iq_j + \dots$, an oscillatory wave function of the form $\expih{g(y)} \in OC^\infty_\hbar(M_2)$, and a bipartite graph $\gamma = (B_\gamma, W_\gamma, E_\gamma)$, we define
    \begin{equation}\label{eq_qpullback_contraction_def}
        \contraction{S_\hbar}{g}{\gamma}(x) := \left( \prod_{v \in B_\gamma}  \frac{m_v!}{(1!)^{k^v_1} (2!)^{k^v_2} \dots (m_v!)^{k^v_{m_v}}} S_\hbar^{e^v_1, \dots, e^v_{m_v}}(x) \right) \left( \prod_{w \in W_t} \dely{e^w_1} \dots \dely{e^w_{m_w}} g(\varphi_\hbar(x)) \right),
    \end{equation}
    where for each $v \in V_\gamma$, we denote the elements of $E_v$, the set of edges connected to $v$, by $e^v_1, \dots, e^v_{m_v}$, and the labels $e^v_k$ are used as dummy indices to be summed over by the Einstein summation convention. Again we let $k^v_i$ be the number of white vertices with which the vertex $v$ shares $i$ edges.  
\end{definition}
Applying Proposition \ref{prop_qgen_eq_graph}, we then get the expansions over bipartite graphs
\begin{align*}
    (\hat{\Phi}^* \expih{g})(x) & = \exp \left( \frac{i}{\hbar} \sum_{\gamma \in [\Gamma^\circ]} \frac{1}{|sym(\gamma)|}  \left(\frac{\hbar}{i}\right)^{b_\gamma}  \contraction{S_\hbar}{g}{\gamma}(x) \right) \\
    & =  \sum_{\gamma \in [\overline{\Gamma}^\circ]} \frac{1}{|sym(\gamma)|}  \left(\frac{\hbar}{i}\right)^{b_\gamma - C_\gamma }  \contraction{S_\hbar}{g}{\gamma}(x) , 
\end{align*}
where we define $C_\gamma$ to be the number of connected components of $\gamma$.

What about pullbacks of general oscillatory wave functions of the form $w(y) = a_\hbar(y)\expih{g(y)}$? Theorem 11 of \cite{Voronov_microformal_homotopyalg} gives the exponential formula for the pullback
\begin{equation}\label{eq_qpullback_exp}
        (\hat{\Phi}^* w)(x) = \expih{S^0_{\hbar}(x)} \left( \expih{S^+_{\hbar}(x,\frac{\hbar}{i}\dely{}) } a_\hbar(y)\expih{g(y)} \right)_{\Big\arrowvert{y^i = \varphi^i_{\hbar}(x)}} ,
\end{equation}
which is the correct analogue of equation \eqref{eq_qgen_eq_exp}. Then using the Leibniz rule we see that all terms resulting from equation \eqref{eq_qpullback_exp} will contain a single factor of $a_\hbar(y)$ acted on by some amount of $y$-derivatives contracted with coefficients of $S^+_\hbar(x,q)$, and the terms coming from derivatives of $\expih{g(y)}$ will be unchanged. So we can represent the $a_\hbar(y)$ terms graphically by a single "special" white vertex which is allowed to have no edges attached to it, representing terms where no derivatives act upon $a_\hbar(y)$.

\subsection{Graphical Calculus for Compositions of Quantum Thick Morphisms}\label{sec_qcomp_graph}

Consider the composition of quantum thick morphisms $\hat{\Phi}_H = \hat{\Phi}_G \circ \hat{\Phi}_F$, with generating function $H_{\hbar}(x,r)$ defined by equation \eqref{eq_qthick_comp}. To apply Proposition \ref{prop_qgen_eq_graph} we perform the same process as in section \ref{sec_comp_graph}, including the use of white-weighted bipartite graphs. Indeed expanding the formal power series $G_\hbar(y,r)$ in $r$, again corresponds to also summing over all possible weightings of the white-vertices, and we get factors of $r$ which contract with the appropriate part of $G$ using the same multi-index $I_\gamma$ from definition \ref{def_multi_index_I}. We now extend the formula in equation \eqref{eq_comp_contraction_def} for the terms in the composition to graphs with loops. 
\begin{definition}
    Given generating functions $F_\hbar(x, q) = F_\hbar^0(x) + \varphi_\hbar^i(x)q_i + F_\hbar^{ij}(x)q_iq_j + \dots$ and
    $G_\hbar(y,r) = G_\hbar^0(y) + \gamma_\hbar^i(y)r_i + G_\hbar^{ij}(y)r_ir_j + \dots$, and a white weighted bipartite graph $\gamma = (B_\gamma, W_\gamma, E_\gamma, L_\gamma)$, we define
    \begin{equation}\label{eq_qcomp_contraction_def}
        \contraction{F_\hbar}{G_\hbar}{\gamma}^{I_\gamma}(x) := \left( \prod_{v \in B_\gamma} \frac{m_v!}{(1!)^{k^v_1} (2!)^{k^v_2} \dots (m_v!)^{k^v_{m_v}}} F_\hbar^{e^v_1, \dots, e^v_{m_v}(x)} \right) \left( \prod_{w \in W_\gamma} \dely{e^w_1} \dots \dely{e^w_{m_w}} G_\hbar^{a_{w, 1} \dots a_{w, \ell_w}}(\varphi_\hbar(x)) \right),
    \end{equation}
    where for each $v \in V_\gamma$, we denote the elements of $E_v$, the set of edges connected to $v$, by $e^v_1, \dots, e^v_{m_v}$, and the labels $e^v_k$ are used as dummy indices to be summed over by the Einstein summation convention. Again we let $k^v_i$ be the number of white vertices with which the vertex $v$ shares $i$ edges.  
\end{definition}

By Proposition \ref{prop_qgen_eq_graph} we get the following expansion for the exponential of the generating function.
\begin{equation*}
    \expih{H_{\hbar}(x,r)} = \exp \left( \frac{i}{\hbar} \sum_{\gamma \in [\Gamma^\circ]_{\infty/2}} \frac{r_{I_\gamma}}{|sym(\gamma)|}  \left(\frac{\hbar}{i}\right)^{b_\gamma}  \contraction{F_\hbar}{G_\hbar}{\gamma}^{I_\gamma}(x) \right)
\end{equation*}
Now taking logarithms and multiplying both sides by $\hbar/i$ we get that the generating function itself has the following expansion as a formal power series in both $r$ and $\hbar$. 
\begin{equation}\label{eq_qcomp_graph}
    H_{\hbar}(x,r) =  \sum_{\gamma \in [\Gamma^\circ]_{\infty/2}} \frac{r_{I_\gamma}}{|sym(\gamma)|}  \left(\frac{\hbar}{i}\right)^{b_\gamma}  \contraction{F_\hbar}{G_\hbar}{\gamma}^{I_\gamma}(x)
\end{equation}

\subsection{Graphical Calculus for Coordinate Transformations of Quantum Thick Morphisms}\label{sec_qtransform_graph}

Consider an invertible change of local coordinates on $M_1\times M_2$ given by $x = x(x')$, $y = y(y')$, with $p', q'$ the corresponding conjugate momenta, under which the generating function for a thick morphism $S(x,q)$ in the "old" coordinates transforms by equation \eqref{eq_qthick_transformation}. To apply Proposition \ref{prop_qgen_eq_graph} we perform the same process as in section \ref{sec_transformation_graph}, including using the same multi-index $\alpha_\gamma$ from definition \ref{def_multi_index_alpha} to write the expansion as a formal power series in $q'$. We now extend the formula in equation \eqref{eq_transformation_contraction_def} for the terms in the coordinate transformation to graphs with loops. 
\begin{definition}
    Given a generating function $S_\hbar(x,q) = S_\hbar^0(x) + \varphi_\hbar^i(x)q_i + S_\hbar^{ij}(x)q_iq_j + \dots$, a change of coordinates $x = x(x')$, $y = y(y')$, and a bipartite graph $\gamma = (B_\gamma, W_\gamma, E_\gamma)$, we define
    \begin{equation}\label{eq_qtransformation_contraction_def}
        \contraction{S_\hbar}{y'}{\gamma}^{\alpha_\gamma}(x') := \left( \prod_{v \in B_\gamma} \frac{m_v!}{(1!)^{k^v_1} (2!)^{k^v_2} \dots (m_v!)^{k^v_{m_v}}} S_\hbar^{e^v_1, \dots, e^v_{m_v}}(x(x')) \right) \left( \prod_{w \in W_\gamma} \dely{e^w_1} \dots \dely{e^w_{m_w}} y^{a_w'}(\varphi_\hbar(x(x'))) \right),
    \end{equation}
    where for each $v \in V_\gamma$, we denote the elements of $E_v$, the set of edges connected to $v$, by $e^v_1, \dots, e^v_{m_v}$, and the labels $e^v_k$ are used as dummy indices to be summed over by the Einstein summation convention. Again we let $k^v_i$ be the number of white vertices with which the vertex $v$ shares $i$ edges.  
\end{definition}

By Proposition \ref{prop_qgen_eq_graph} we get an expansion for the exponential of the generating function $S_\hbar'(x',q')$ in the "new" coordintates. Then by taking logarithms we get that $S_\hbar'(x',q')$ itself has the following expansion as a formal power series in $q'$
\begin{equation}
    S_\hbar'(x',q') = \sum_{\gamma \in [\Gamma^\circ]} \frac{q'_{\alpha_\gamma}}{|sym(\gamma)|} \contraction{S_\hbar}{y'}{\gamma}^{\alpha_\gamma}(x') .
\end{equation}

\addcontentsline{toc}{section}{Discussion and Example Calculation}
\section*{Discussion and Example Calculation}\label{sec_discussion}

\addcontentsline{toc}{subsection}{The Classical Limit}
\subsection*{The Classical Limit}
As expected, the graphical calculus for classical thick morphisms is the $\hbar \rightarrow 0$ limit of the graphical calculus for quantum thick morphisms. In \cite{Voronov_microformal_homotopyalg}, Voronov shows that classical pullbacks, compositions, and coordinate transformations are all the $\hbar \rightarrow 0$ limits of their quantum counterparts. Graphically, taking the $\hbar \rightarrow 0$ limit amounts to ignoring graphs $\gamma$ with $b_\gamma > 0$, i.e. restricting to bipartite trees. So we recover the classical graphical calculus in this way.

\addcontentsline{toc}{subsection}{Removing the Assumption of Symmetry for the Generating Function}
\subsection*{Removing the Assumption of Symmetry for the Generating Function}

Next we will discuss what happens to the graphical calculus when the coefficients $S^{a_1 \dots a_m}(x)$ are no longer assumed to be symmetric in their indices. It is sufficient to look at the quantum case and then take the classical limit to obtain the classical case. The complication is that Feynman rules must now specify to which index in $S^{a_1 \dots a_m}(x)$ an edge in a bipartite graph corresponds to a contraction of. So we must upgrade our class of bipartite graphs so that the set of edges from each black vertex $v$ is equipped with a total order or equivalently a labeling by $1, \dots, m_v$, where $m_v$ is the degree of $v$. The correct notion of isomorphism for such graphs is given by isomorphisms for bipartite graphs where the bijection on edges respects the edge ordering for each black vertex. We then upgrade the Feynman rules so that the $k$th edge from some black vertex to a white vertex corresponds to contraction in the $k$th index of the associated coefficient of $S$. Also, for any black vertex $v$, we change the factor given in \eqref{eq_Feynman_rules_extension} by simply removing the combinatorial factor in front of $S^{* \dots *}$, since this was included to account for different orderings of contractions with $S^{* \dots *}$ that result in the same graph. But, since the edges of our graphs are now labeled with this ordering, these different orderings will produce different graphs in the resulting expansion. With these modifications all the expansions given for the quantum case (and by taking $\hbar \rightarrow 0$ the classical case) take the same form but with sums over all graphs with edge sets from each black vertex equipped with total orders, denoting this set by $\Gamma^\circ_{ord}$ and the corresponding set of trees by $T^\circ_{ord}$.

\addcontentsline{toc}{subsection}{The Super Case}
\subsection*{The Super Case}

We now allow the source and target of our thick morphisms to be supermanifolds. The existence of odd (anti-commuting) variables introduces a number of complications to do with signs and ordering which we will deal with in turn. Again we will first look at the quantum case and derive the classical case by taking $\hbar \rightarrow 0$. 

Firstly, while the indices on coefficients $S^{a_1 \dots a_m}(x)$ can be assumed to be symmetric in a super sense (swapping adjacent indices gives the usual sign dependent on the parities of these indices), this does not aid the calculation in the quantum case. So, as in the previous section, we do not assume symmetry and use graphs with ordered sets of edges coming from each black vertex.

Next, notice that the coefficients $S^{a_1 \dots a_m}(x)$ can themselves have odd parity, since the generating function must be even. Explicitly, $S^{a_1 \dots a_m}(x)$ must be odd if and only if the total parity of $q_{a_1} \dots q_{a_m}$ is odd, i.e. $a_1 + \dots + a_m = 1 \mod 2$. So the coefficients $S^{a_1 \dots a_m}(x)$ do not necessarily commute with each other and their order must be specified by the Feynman rules. This means that the black vertices of our graphs must have a fixed order, and we cannot use topological bipartite graphs.

Now consider the proof of lemma \ref{lemma_qgen_eq_initial}. The first problem arises at the calculation of the general term \eqref{eq_calc_genterm}. Since the general term now involves derivatives by odd variables we need a super version of the Fa\`{a} di Bruno formula. We will show a general version of such a formula in an upcoming paper. In this case, since $\frac{i}{\hbar}G(y)$ is even, the Fa\`{a} di Bruno formula is the same as in the even case, except with some additional signs and a specific ordering of the factors and derivatives, which are taken care of by the following definitions.

\begin{definition}\label{def_partition_ordering}
    Let $\pi$ be a partition of the ordered set of indices $\{ a_1, \dots, a_m \}$. Then each block inherits the ordering of the whole set, and we order the blocks in $\pi$ by their last element, denoting them in order by $B^1_\pi, \dots, B^{|\pi|}_\pi$. The sum of the ordering on blocks of $\pi$ with the ordering within blocks also defines a new ordering on the whole set of indices, and we denote this ranking by $\pi(a_i)$ for any $i = 1, \dots, m$. For example, let $m=5$ and take the partition $ \pi = \{ \{a_1,a_4\},\{a_2,a_5\},\{a_3\}\}$. Then we have $B^1_\pi = \{a_3\}$, $B^2_\pi = \{a_1, a_4\}$, and $B^3_\pi = \{a_2, a_5\}$, and the ranking on indices is given by $\pi(a_3) < \pi(a_1) < \pi(a_4) < \pi(a_2) < \pi(a_5)$. 
\end{definition}
\begin{definition}\label{def_partition_parity}
    We define the parity $\tilde{\pi}$ of any partition $\pi$ of $\{ a_1, \dots, a_m \}$ by
    \begin{equation}
        \tilde{\pi} \coloneqq \sum_{\substack{1 \le i,j \le m \\ i<j \,\land\, \pi(i) > \pi(j)}} \tilde{a_i} \tilde{a_j} \mod 2.
    \end{equation}
    The value of $\tilde{\pi}$ essentially measures the parity of the distance of the ordering of the partition from the original ordering of the set of indices, where only two disordered indices both having odd parity contribute. We can think of this as a special version of the Kendall tau distance between the original ordering and the ordering induced by $\pi$ on the set of indices.
\end{definition}  
The super Fa\`{a} di Bruno formula then is the same as the classical case but the term associated to each partition $\pi$ comes with a sign $(-1)^{\tilde{\pi}}$, and the derivatives over blocks $\dely{(B)}$ are ordered with respect to the ordering on blocks, and the individual partial derivatives in the operator $\dely{(B)}$ are ordered by the ordering within the blocks, as described in definition \ref{def_partition_ordering}. In particular equation \eqref{eq_fdbcalc} becomes 
\begin{equation}\label{eq_fdbcalc_super}
        \left( \dely{a_{1,1}} \dots \dely{a_{1,m_1}} \right) \dots \left( \dely{a_{n,1}} \dots \dely{a_{n,m_n}} \right) \expih{G(y)} = \sum_{\pi \in P} (-1)^{\tilde{\pi}} \left(  \prod_{i = 1}^{|\pi|} \frac{i}{\hbar} \dely{(B^i_\pi)} G(y) \right) \expih{G(y)} .
\end{equation}
Graphically all such terms are again described by an edge labeled bipartite graph as before, and the sign associated to a graph can be pictorially calculated by the following procedure.
\begin{definition}
    To each edge labeled bipartite graph $\gamma$ we associate a value $\tilde{\gamma} \in \mathbb{Z}_2$ by:
    \begin{enumerate}
        \item Draw the sets of black and white vertices in $\gamma$ in order along two parallel lines - note that the order of the white vertices is defined by the order of the highest order edge attached to them
        \item For each edge $e \in E_\gamma$, first draw a "waypoint" $w_e$ along a third parallel line in between the black and white vertices. The waypoints should be ordered by the total order on the entire set of edges.
        \item For each edge $e \in E_\gamma$, draw a line from $e_\bullet$ to the associated waypoint $w_e$ and then from $w_e$ to $e_\circ$.
        \item Edges may cross in the area between the waypoints and the white vertices. For each crossing of two edges $e$ and $e'$, associate a parity given by the product $\tilde{a}_e \tilde{a}_{e'}$ of the parities associated to the indices for each edge.
        \item Finally obtain $\tilde{\gamma}$ as the sum over all of the parities for each crossing of edges.
    \end{enumerate}
    For example take the graph drawn below according to our procedure.
    \[\begin{tikzpicture}[x=0.67cm, y=0.5cm]
    \vertex[fill, label=left: 2] (b2) at (0,0.5) {};
    \vertex[fill, label=left: 1] (b1) at (0,2.5) {};

    \vertex[inner sep = 1pt] (a) at (1,3) {a};
    \vertex[inner sep = 1pt] (b) at (1,2) {b};
    \vertex[inner sep = 1pt] (c) at (1,1) {c};
    \vertex[inner sep = 1pt] (d) at (1,0) {d};

    \vertex (w2)[label=right: 2] at (3,0.5) {};
    \vertex (w1)[label=right: 1] at (3,2.5) {};
    
	\path
		(b1) edge (a)
        (b1) edge (b)
        (b2) edge (c)
        (b2) edge (d)

        (a) edge (w2)
        (b) edge (w1)
        (c) edge (w1)
        (d) edge (w2)
	;
    \end{tikzpicture}\]
    Notice there are two crossings, between the edges labeled by a and b, and between the edges labeled by a and c. So its total associated term (with sign), for the general expansion, is given by
    \begin{equation*}
        (-1)^{\tilde{a}\tilde{b} + \tilde{a}\tilde{c}} \frac{\hbar}{i}S^{ab}S^{cd}\dely{b}\dely{c}G(y)\dely{a}\dely{d}G(y) .
    \end{equation*}
\end{definition}
Now notice that if $\gamma$ is a disconnected graph with no mixing in the ordering of the vertices in each connected component, then the sign $(-1)^{\tilde{\gamma}}$ will be the product of the signs of each connected component. For this reason, the calculations in the proof of Proposition \ref{prop_qgen_eq_graph} all go through, and we obtain a super version of the general expansion. Explicitly, the super version of the integral $I(S,G)$ has the expansion
\begin{equation}\label{eq_qgen_eq_graph_super}
        I(S,G) = \exp \left( \frac{i}{\hbar} \sum_{\gamma \in \Gamma^\circ_{ord}} \frac{(-1)^{\tilde{\gamma}}}{|\gamma|_\bullet! |\gamma|_\circ!}  \left(\frac{\hbar}{i}\right)^{b_\gamma}  \contraction{S}{G}{\gamma} \right) .
\end{equation}
Again, the three cases of pullbacks, composition, and coordinate transformations all follow from this expansion in the same way as before, and the classical calculi follow by taking $\hbar \rightarrow 0$.

\addcontentsline{toc}{subsection}{Example Calculation}
\subsection*{Example Calculation}
To finish we use the graphical calculus for quantum pullbacks to calculate a low order approximation of a general pullback.  

\begin{example}
    Let's calculate the logarithm of the pullback of $\expih{g(y)}$ by $\hat{\Phi}$ up to third order in $g$ and first order in $\hbar$. The graphs involved are all white-leaved bipartite graphs with at most 3 white vertices and at most 1 loop, which we list below.
    
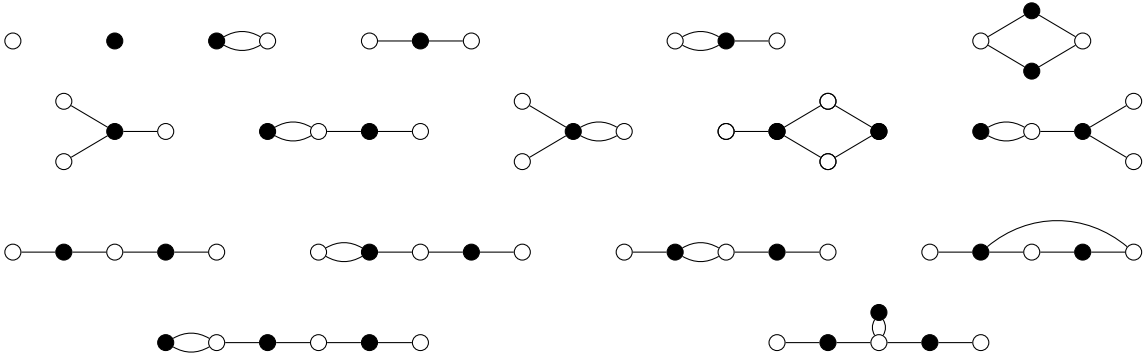
\[\begin{tikzpicture}[x=0.67cm, y=0.5cm]
	\vertex (a1) at (0,0) {};
    \vertex[fill] (a2) at (1,0) {};
    \vertex (a3) at (2,0) {};
    \vertex[fill] (a4) at (3,0) {};
    \vertex (a5) at (4,0) {};
    
	\vertex (b1) at (6,0) {};
    \vertex[fill] (b2) at (7,0) {};
    \vertex (b3) at (8,0) {};
    \vertex[fill] (b4) at (9,0) {};
    \vertex (b5) at (10,0) {};
    
	\vertex (c1) at (12,0) {};
    \vertex[fill] (c2) at (13,0) {};
    \vertex (c3) at (14,0) {};
    \vertex[fill] (c4) at (15,0) {};
    \vertex (c5) at (16,0) {};
    
	\vertex (d1) at (18,0) {};
    \vertex[fill] (d2) at (19,0) {};
    \vertex (d3) at (20,0) {};
    \vertex[fill] (d4) at (21,0) {};
    \vertex (d5) at (22,0) {};
    
    \vertex (e1) at (1,3) {};
    \vertex (e2) at (1,5) {};
    \vertex[fill] (e3) at (2,4) {};
    \vertex (e4) at (3,4) {};

    \vertex[fill] (n1) at (5,4) {};
    \vertex (n2) at (6,4) {};
    \vertex[fill] (n3) at (7,4) {};
    \vertex (n4) at (8,4) {};
    
    \vertex (f1) at (10,3) {};
    \vertex (f2) at (10,5) {};
    \vertex[fill] (f3) at (11,4) {};
    \vertex (f4) at (12,4) {};
    
    \vertex (g1) at (14,4) {};
    \vertex[fill] (g2) at (15,4) {};
    \vertex (g3) at (16,3) {};
    \vertex (g4) at (16,5) {};
    \vertex[fill] (g5) at (17,4) {};

    \vertex[fill] (o1) at (19,4) {};
    \vertex (o2) at (20,4) {};
    \vertex[fill] (o3) at (21,4) {};
    \vertex (o4) at (22,3) {};
    \vertex (o5) at (22,5) {};

    \vertex (g1) at (14,4) {};
    \vertex[fill] (g2) at (15,4) {};
    \vertex (g3) at (16,3) {};
    \vertex (g4) at (16,5) {};
    \vertex[fill] (g5) at (17,4) {};
    
    \vertex (h1) at (0,7) {};
    
    \vertex[fill] (i1) at (2,7) {};

    \vertex[fill] (m1) at (4,7) {};
    \vertex (m2) at (5,7) {};
    
    \vertex (j1) at (7,7) {};
    \vertex[fill] (j2) at (8,7) {};
    \vertex (j3) at (9,7) {};
    
    \vertex (k1) at (13,7) {};
    \vertex[fill] (k2) at (14,7) {};
    \vertex (k3) at (15,7) {};
    
    \vertex (l1) at (19,7) {};
    \vertex[fill] (l2) at (20,6) {};
    \vertex[fill] (l3) at (20,8) {};
    \vertex (l4) at (21,7) {};

    \vertex[fill] (q1) at (3,-3) {};
    \vertex (q2) at (4,-3) {};
    \vertex[fill] (q3) at (5,-3) {};
    \vertex (q4) at (6,-3) {};
    \vertex[fill] (q5) at (7,-3) {};
    \vertex (q6) at (8,-3) {};

    \vertex (p1) at (15,-3) {};
    \vertex[fill] (p2) at (16,-3) {};
    \vertex (p3) at (17,-3) {};
    \vertex[fill] (p4) at (17,-2) {};
    \vertex[fill] (p5) at (18,-3) {};
    \vertex (p6) at (19,-3) {};
    
	\path
		(a1) edge (a2)
        (a2) edge (a3)
        (a3) edge (a4)
        (a4) edge (a5)

        (b1) edge[bend left = 30] (b2)
        (b1) edge[bend right = 30] (b2)
        (b2) edge (b3)
        (b3) edge (b4)
        (b4) edge (b5)

        (c1) edge (c2)
        (c2) edge[bend left = 30] (c3)
        (c2) edge[bend right = 30] (c3)
        (c3) edge (c4)
        (c4) edge (c5)

        (d1) edge (d2)
        (d2) edge (d3)
        (d3) edge (d4)
        (d4) edge (d5)
        (d2) edge[bend left = 40] (d5)

        (e1) edge (e3)
        (e2) edge (e3)
        (e3) edge (e4)

        (f1) edge (f3)
        (f2) edge (f3)
        (f3) edge[bend left = 30] (f4)
        (f3) edge[bend right = 30] (f4)

        (g1) edge (g2)
        (g2) edge (g3)
        (g2) edge (g4)
        (g3) edge (g5)
        (g4) edge (g5)

        (m1) edge[bend left = 30] (m2)
        (m1) edge[bend right = 30] (m2)

        (j1) edge (j2)
        (j2) edge (j3)

        (k1) edge[bend left = 30] (k2)
        (k1) edge[bend right = 30] (k2)
        (k2) edge (k3)

        (l1) edge (l2)
        (l1) edge (l3)
        (l2) edge (l4)
        (l3) edge (l4)

        (n1) edge[bend left = 30] (n2)
        (n1) edge[bend right = 30] (n2)
        (n2) edge (n3) 
        (n3) edge (n4)

        (o1) edge[bend left = 30] (o2)
        (o1) edge[bend right = 30] (o2)
        (o2) edge (o3)
        (o3) edge (o4)
        (o3) edge (o5)

        (q1) edge[bend left = 30] (q2)
        (q1) edge[bend right = 30] (q2)
        (q2) edge (q3)
        (q3) edge (q4)
        (q4) edge (q5)
        (q5) edge (q6)

        (p1) edge (p2)
        (p2) edge (p3) 
        (p3) edge[bend left = 30] (p4)
        (p3) edge[bend right = 30] (p4)
        (p3) edge (p5)
        (p5) edge (p6)
	;
    \end{tikzpicture}\]
\end{example}
By summing over these graphs we get the approximation
\begin{gather*}
    (\hat{\Phi}^* \expih{g})(x) \approx \exp \left[ \frac{i}{\hbar}  \left(  g(\varphi_\hbar(x)) + S^0_\hbar(x) + \frac{\hbar}{i}S^{ab}_\hbar(x) \dely{a}\dely{b}g(\varphi_\hbar(x))  + S^{ab}_\hbar(x) \dely{a}g(\varphi_\hbar(x)) \dely{b}g(\varphi_\hbar(x)) \right.\right. \\ 
    + 3 \frac{\hbar}{i}S^{abc}_\hbar(x) \dely{a}\dely{b}g(\varphi_\hbar(x))\dely{c}g(\varphi_\hbar(x)) + \frac{\hbar}{i}S^{ab}_\hbar(x)S^{cd}_\hbar(x)\dely{a}\dely{b}g(\varphi_\hbar(x))\dely{c}\dely{d}g(\varphi_\hbar(x))\\
    + S^{abc}_\hbar(x)\dely{a}g(\varphi_\hbar(x))\dely{b}g(\varphi_\hbar(x))\dely{c}g(\varphi_\hbar(x)) 
    + 2 \frac{\hbar}{i}S^{ab}_\hbar(x)S^{cd}_\hbar(x)\dely{a}\dely{b}\dely{c}g(\varphi_\hbar(x))\dely{d}g(\varphi_\hbar(x)) \\
    + 6\frac{\hbar}{i}S^{abcd}_\hbar(x)\dely{a}g(\varphi_\hbar(x))\dely{b}g(\varphi_\hbar(x))\dely{c}\dely{d}g(\varphi_\hbar(x)) \\ 
    + 6 \frac{\hbar}{i}S^{abc}_\hbar(x)S^{de}_\hbar(x)\dely{a}g(\varphi_\hbar(x))\dely{b}\dely{d}g(\varphi_\hbar(x))\dely{c}\dely{e}g(\varphi_\hbar(x)) \\
    + 3\frac{\hbar}{i}S^{ab}_\hbar(x)S^{cde}_\hbar(x)\dely{a}\dely{b}\dely{c}g(\varphi_\hbar(x))\dely{d}g(\varphi_\hbar(x))\dely{e}g(\varphi_\hbar(x)) \\ 
    + 2S^{ab}_\hbar(x)S^{cd}_\hbar(x)\dely{a}g(\varphi_\hbar(x))\dely{b}\dely{c}g(\varphi_\hbar(x))\dely{d}g(\varphi_\hbar(x)) \\
    + 6\frac{\hbar}{i}S^{abc}_\hbar(x)S^{de}_\hbar(x)\dely{a}\dely{b}g(\varphi_\hbar(x))\dely{c}\dely{d}g(\varphi_\hbar(x))\dely{e}g(\varphi_\hbar(x)) \\ 
    + 6\frac{\hbar}{i}S^{abc}_\hbar(x)S^{de}_\hbar(x)\dely{a}g(\varphi_\hbar(x))\dely{b}\dely{c}\dely{d}g(\varphi_\hbar(x))\dely{e}g(\varphi_\hbar(x)) \\
    + 12\frac{\hbar}{i}S^{abc}_\hbar(x)S^{de}_\hbar(x)\dely{a}g(\varphi_\hbar(x))\dely{b}\dely{d}g(\varphi_\hbar(x))\dely{c}\dely{e}g(\varphi_\hbar(x)) \\ 
    + 4\frac{\hbar}{i}S^{ab}_\hbar(x)S^{cd}_\hbar(x)S^{ef}_\hbar(x)\dely{a}\dely{b}\dely{c}g(\varphi_\hbar(x))\dely{d}\dely{e}g(\varphi_\hbar(x))\dely{f}g(\varphi_\hbar(x)) \\
    + 4\frac{\hbar}{i}S^{ab}_\hbar(x)S^{cd}_\hbar(x)S^{ef}_\hbar(x)\dely{c}g(\varphi_\hbar(x))\dely{a}\dely{b}\dely{d}\dely{e}g(\varphi_\hbar(x))\dely{f}g(\varphi_\hbar(x))
    \left.\left.            \vphantom{\frac{i}{\hbar}}\right)\right] .  
\end{gather*}

\addcontentsline{toc}{section}{Declarations}
\section*{Declarations}
    This work was supported by the Additional Funding Programme for Mathematical Sciences, delivered by EPSRC (EP/V521917/1) and the Heilbronn Institute for Mathematical Research. The author has no competing interests to declare that are relevant to the content of this article. This version of the article has been accepted for publication after peer review, but is not the Version of Record. The version of record is available online at \url{https://doi.org/10.1007/s00220-025-05372-9}.

\addcontentsline{toc}{section}{References}
\printbibliography

\end{document}